\documentclass{amsart}
\usepackage[utf8]{inputenc}
\usepackage{amsmath}
\usepackage{amssymb}
\usepackage{verbatim}
\usepackage{geometry}
\usepackage{tikz}
\usepackage{esint}
\usepackage{soul}
\usepackage[colorlinks=true,urlcolor=blue,citecolor=red,linkcolor=blue,linktocpage,pdfpagelabels,bookmarksnumbered,bookmarksopen]{hyperref} 
\usepackage[hyperpageref]{backref}

\date{\today}

\newtheorem{thm}{Theorem}[section]
\newtheorem{cor}[thm]{Corollary}
\newtheorem{prop}[thm]{Proposition}
\newtheorem{lemma}[thm]{Lemma}
\theoremstyle{definition}
\newtheorem{definition}[thm]{Definition}

\newtheorem{rem}[thm]{Remark}
\newtheorem*{ack}{Acknowledgments}

\newcommand{\weakst}{\stackrel{\ast}{\rightharpoonup}}

\newcommand\N{{\mathbb N}}
\newcommand\Z{{\mathbb Z}}
\newcommand\Q{{\mathbb Q}}
\newcommand\R{{\mathbb R}}

\newcommand\A{{\mathcal A}}

\numberwithin{equation}{section}
\let\wto\rightharpoonup
\let\phi\varphi

\title[Stochastic homogenization of nonconvex integrals with Orlicz growth]{Stochastic homogenization of nonconvex unbounded integral functionals with generalized Orlicz growth} 
\author[Aruta]{Davide Aruta}
\address[D. Aruta]{Scuola Superiore Meridionale,
\newline\indent via Mezzocannone 4, 80134 Napoli, Italy}

\email{d.aruta@ssmeridionale.it}
\author[Prinari]{Francesca Prinari}

\address[F. Prinari]{Dipartimento di Scienze Agrarie, Alimentari e Agro-ambientali
\newline\indent 
Universit\`a di Pisa
\newline\indent
Via del Borghetto 80, 56124 Pisa, Italy}
\email{francesca.prinari@unipi.it}

\author[Solombrino]{Francesco Solombrino}
\address[F. Solombrino]{Dipartimento di Scienze e Tecnologie Biologiche ed Ambientali
\newline\indent 
 Università del Salento, 
 \newline\indent
Complesso Ecotekne, via Monteroni 165, 73100 Lecce, Italy}
\email{francesco.solombrino@unisalento.it}


\begin{document}

\begin{abstract}
We consider the homogenization of random integral functionals which are possibly unbounded, that is, the domain of the integrand is not the whole space and may depend on the space-variable. In the vectorial case, we develop a complete  stochastic homogenization theory for nonconvex unbounded functionals with convex growth of generalized Orlicz-type, under a standard set of assumptions in the field, in particular a coercivity condition of order $p^->1$, and an upper bound of order $p^+<\infty$.
The limit energy is defined in a possibly anisotropic Musielak-Orlicz space, for which approximation results with smooth functions are provided. The proof is based on the localization method of $\Gamma$-convergence and a careful use of truncation arguments.
\end{abstract}

\maketitle

\begin{center}
\begin{minipage}{10cm}
\small
\tableofcontents
\end{minipage}
\end{center}

\section{Introduction}
Homogenization theory addresses the passage from microscopic to macroscopic descriptions of heterogeneous media. Within the calculus of variations, $\Gamma$-convergence provides a natural framework for homogenization, as it guarantees convergence of minimizers and minimal values and leads to an intrinsic characterization of effective energies. A systematic exposition of this approach for variational integrals is given in the monograph by Braides and Defranceschi~\cite{BDF}.

\noindent{\bf General variational setting.}
The prototypical objects of study are families of functionals
\[
F_\varepsilon(u)
:=
\int_U f\!\left(\frac{x}{\varepsilon}, D u(x)\right)\,dx,
\qquad
u\in W^{1,p}(U,\mathbb{R}^N),
\]
where $f(x,\Sigma)$ is a Carath\'eodory integrand oscillating on the microscopic scale $\varepsilon>0$.
In the \emph{periodic case}, $f(\cdot,\Sigma)$ is periodic in $x$. In the \emph{stochastic case}, one considers a probability space $(\Omega,\mathcal{F},\mathbb{P})$ endowed with a measure-preserving ergodic group action $(\tau_z)_{z\in\mathbb{Z}^d}$ and integrands which comply with the stochastic periodicity condition (stationarity)
\[
f(\omega,x+z,\Sigma)=f(\tau_z\omega, x,\Sigma), \quad \hbox{ for every } (\omega,z,\Sigma)\in \Omega\times \Z^d\times {\bf M}^{N\times d}, \hbox{ for a.e. } x\in U . 
\]
Under appropriate assumptions, the family $(F_\varepsilon)$ $\Gamma$-converges as $\varepsilon\to0$ to a deterministic functional
\[
F_{\mathrm{hom}}(u)
=
\int_U f_{\mathrm{hom}}(D u(x))\,dx,
\]
where $f_{\mathrm{hom}}$ is the homogenized density.
A cornerstone of the theory is the integral representation theorem for $\Gamma$-limits of equi-coercive families of integral functionals. If $(F_\varepsilon)$ satisfies suitable locality, lower semicontinuity, and coercivity assumptions, then any $\Gamma$-limit admits a representation of the form
\[
F(u)
=
\int_\Omega f_{\mathrm{hom}}(x,D u(x))\,dx.
\]
This structural property reduces homogenization to the identification of the effective integrand $f_{\mathrm{hom}}$, typically characterized through asymptotic minimization problems (cell formulas) or ergodic arguments.
The works of Dal Maso and Modica~\cite{dmmodica} established a rigorous $\Gamma$-convergence framework for stochastic homogenization of integral functionals under standard $p$-growth conditions. 
For $\Sigma\in\mathbf{M}^{N\times d}$ and a bounded Lipschitz set $A\subset\mathbb{R}^d$, they consider the random set function
\[
\mu(A,\omega)
:=
\inf\left\{
\int_A f(\omega,x,\Sigma+D \varphi(x))\,dx
:
\varphi\in W^{1,p}_0(A,\mathbb{R}^N)
\right\}.
\]
The mapping $A\mapsto \mu(A,\omega)$ is subadditive and stationary. The Akçoglu--Krengel subadditive ergodic theorem~\cite{AK} then yields, for cubes $Q_R$,
\[
\lim_{R\to\infty}
\frac{1}{|Q_R|}
\mu(Q_R,\omega)
=
f_{\mathrm{hom}}(\Sigma)
\quad\text{for }\mathbb{P}\text{-a.e. }\omega,
\]
with a deterministic limit $f_{\mathrm{hom}}(\Sigma)$. This result ensures the existence of the homogenized density and explains the deterministic macroscopic behavior despite microscopic randomness. Later, in \cite{Messaoudi-Michaille-94} Messaoudi and Michaille  treated the homogenization of \emph{quasiconvex stationary ergodic} integrands satisfying a \emph{polynomial standard growth condition}, following Dal Maso and Modica's approach.

\noindent{\bf Beyond standard growth.}
Classical homogenization assumes two-sided polynomial $p$-growth. Recent years, however, witnessed an increasing interest for non-standard growth conditions. They account for the case of unbounded integrands (\cite{carbone}), that is the finiteness domain is not fixed and may depend on the space variable. A starting point is represented by the results obtained by Braides and Defranceschi, who treated the case of \emph{nonconvex periodic} integrands (see also~\cite{CorboEsposito-DeArcangelis-92} in the convex case) where $f$ is periodic in $x$ and satisfies the \emph{polynomial non-standard growth condition}
\begin{equation*}
\frac{1}{C}|\Sigma|^p-C\,\leq\, f(x,\Sigma)\,\leq \, C\left(1+|\Sigma|^q\right)
\end{equation*}
for some $q<p^*$ (with $p^*$ the Sobolev-conjugate of $p>1$), and the \emph{doubling property} (also called \emph{$\Delta_2$ condition})
\begin{equation}\label{eq:intro-doubling}
 f(x,2\Sigma)\le C(1+f(x, \Sigma)).
\end{equation}
Their results were also generalized  by Zhikov, Kozlov and Oleinik  \cite[Chapter~15]{ZKO94} to the case of \emph{convex stationary ergodic} integrands.
More recently, Gloria and Duerinckx~\cite{GloriaDuerinckx} extended stochastic homogenization to integrands that may be nonconvex and unbounded from above, provided they satisfy a convex coercive lower bound of order $p>d$. Their result shows that $\Gamma$-convergence homogenization remains valid well beyond uniformly bounded settings.
Ruf and Schäffner~\cite{RS23} established stochastic homogenization for convex integral functionals \emph{without any prescribed growth condition}. Instead of polynomial bounds, their approach relies on convexity and a structural property allowing componentwise truncations of competitors without uncontrolled increase of the energy. This significantly broadens the class of admissible integrands and further demonstrates the robustness of the $\Gamma$-convergence framework. They also show, that, if one renounces to properties allowing for componentwise truncation, the coercitivity lower bound in \cite{GloriaDuerinckx} can be improved to $p>d-1$. This leads us to the purpose of the present paper.

\noindent{\bf Description of our results.}
In this paper we are concerned with the periodic and stochastic homogenization in the vector-valued setting of possibly nonconvex unbounded integrands with a convex growth of the form
\begin{equation}\label{intro: control}
\frac{1}{C}g(x, |\Sigma|)\,\leq\, f(x,\Sigma)\,\leq \, C(1+g(x, |\Sigma|))
\end{equation}
where $g$ satisfies the customary assumptions in the setting of the theory of generalized Orlicz-Sobolev spaces (see \cite{HarHas}).  In particular, a coercivity condition of order $p^->1$ and an upper bound of order $p^+<\infty$ are assumed; differently from \cite{BDF}, no relationship between $p^-$ and $p^+$ is imposed, while, in comparison to \cite{GloriaDuerinckx} and \cite{RS23}, the coercivity order can be lowered arbitrarily close to $1$.

The underlying idea is to build up on the model case $f=g$: here, as $g$ is radially increasing and convex, one can show that the homogenized density $\gamma$ (which needs not be radial itself) inherits a stronger form of the doubling condition \eqref{eq:intro-doubling}, tailored to matrix multiplication and  taking the form
\begin{equation}\label{eq:intro-strongdoubling}
\gamma(A\Sigma)\le C_A\gamma(\Sigma)\,,
\end{equation}
where the costant $C_A$ is uniformly bounded when $A$ varies in a compact subset of $\mathbf{M}^{N\times d}$. It allows one to deal with smooth truncations in the vector-valued setting. In particular, it is also used to prove an approximation result with smooth functions (Theorem \ref{approximgamma} and Corollary \ref{senzaLinfty}), which generalizes the result of \cite[Lemma 3.6]{muller} for strongly star-shaped domains to the case of general Lipschitz domains. With this, the Musielak-Orlicz-Sobolev space $W_\gamma(U,\mathbb R^N)$ of functions $u\in W^{1,1}(U,\mathbb R^N)$ satisfying 
\[
\int_U \gamma(Du(x))\,\mathrm{d}x<\infty
\]
builds the ideal framework for an integral representation result for local functionals (Theorem \ref{repr}) and to apply the localization method of $\Gamma$-convergence \cite[Chapters 18-20]{DM} to analyse the homogenized limit of the functionals $F_\varepsilon$. A key tool to this aim is the subadditivity estimate in Lemma \ref{subadditivityF}. Its proof makes use of smooth truncations, and with the help of \eqref{intro: control} and \eqref{eq:intro-strongdoubling} reproduces in our setting an idea which has already been used in the context of the homogenization in the vector-valued setting, for instance, in \cite[Theorem 4.3]{CDSZ}.

Our main result is Theorem \ref{main}, that states (under \eqref{intro: control} and the usual stationarity and ergodicity assumptions) the $\Gamma$-convergence of the random integral functionals
\[
F_{\varepsilon}(\omega, u) =  
\displaystyle \int_U f\left(\omega, \frac{x}{\varepsilon}, Du(x)\right)\,\mathrm{d}x
\]
to a functional of the form
\[
F( u)=\int_U \varphi( Du(x))\,\mathrm{d}x
\]
with domain $W_\gamma(U,\mathbb R^N)$, where $\gamma$ is the density corresponding to the case $f=g$.

\bigskip

\noindent{\bf Outline of the paper.} The paper is organized as follows. In Section \ref{notaz} we fix the basic notation and definitions that are needed throughout the paper. Section \ref{secmusielak} presents the required background about Musielak-Orlicz spaces, together with a technical lemma that is useful for truncation arguments. Section \ref{convex homog} deals with the case $f=g$ of the deterministic homogenization of integral functionals with convex, radially increasing densities. Under the extra assumption of periodicity in the space variable, we also characterize the limit density through a cell formula. In Section \ref{nonconvexcase} we prove a result of deterministic homogenization in the nonconvex case. In Section \ref{stochastic} we prove two results of stochastic homogenization, one in the convex case and the other in the nonconvex setting, that is the main result of the paper.

\begin{ack}
The authors are   members of the {\it Gruppo Nazionale per l'Analisi Matematica, la Probabilit\`a
e le loro Applicazioni} (GNAMPA) of the Istituto Nazionale di Alta Matematica (INdAM).  Francesca Prinari thanks the  Dipartimento di Matematica e Fisica  “E. De Giorgi” in Lecce for its kind hospitality.
\end{ack}
\section{Notation and Background}\label{notaz}

Let $U\subseteq \mathbb R^d$ be a bounded open set. We denote by $\A({U })$ the class of the open subsets of $U$ and by $\A_0({U })$ the class of the Lipschitz open subsets of $U$, respectively. 
Let $E_{\varepsilon}\colon L^1({U },\R^N)\times\A({U })\to [0,+\infty]$ be a family of integral functionals
of the form \begin{equation*}
    E_{\varepsilon}(u,V)=\int_V e\left(\frac x \varepsilon, Du\right)\,\mathrm{d}x.
\end{equation*}
In order to prove our $\Gamma$-convergence results, we will often use a localization procedure, hence we need the following definition of localized $\Gamma$-liminf and $\Gamma$-limsup. Given a subsequence ${\varepsilon_k}\to 0$, for every $V\in \mathcal A({U })$ we define  \begin{align*}
    E'(u,V)=\Gamma(L^1)\text{-}\liminf_{k\to +\infty} E_{\varepsilon_k}(u,V)=\inf\{\liminf_{k\to +\infty}E_{\varepsilon_k}(u_k,V)\mid u_k\to u \text{ in } L^1(V,\R^N)\}, \\
E''(u,V)=\Gamma(L^1)\text{-}\limsup_{k\to +\infty} E_{\varepsilon_k}(u,V)=\inf\{\limsup_{k\to +\infty}E_{\varepsilon_k}(u_k,V)\mid u_k\to u \text{ in } L^1(V,\R^N)\}.
\end{align*}
Notice that $E'$ and $E''$ may depend on the subsequence $\varepsilon_k$.

We remark that $E_{\varepsilon_k}$ $\Gamma$-converges to $E$ in the strong topology of $L^1({U },\R^N)$ is equivalent to $E'(u,{U })=E''(u,{U })=E(u,{U })$ for every $u\in L^1({U },\R^N)$, in which case we write 
\begin{equation*}
\Gamma(L^1)\text{-}\lim_{k\to +\infty} E_{\varepsilon_k}=E. 
\end{equation*} The Urysohn property of $\Gamma$-convergence guarantees that, if $X$ is a first-countable topological space (for instance, a metric space), $E_{\varepsilon_k}$ $\Gamma$-converges to $E$ in the topology of $X$ if every subsequence of $E_{\varepsilon_k}$ has a further subsequence that $\Gamma$-converges to $E$ (see \cite[Proposition 8.3]{DM}).

In the following,  $\mathcal B({U }) $ will be the $\sigma$-algebra of the Borel subsets of $U$.  For every Lebesgue measurable set  $B\subseteq \R^d$,  we denote by $|B|$  the Lebesgue measure of $B$.
Moreover,  we will say that a subset $\mathcal D\subseteq \A_0({U })$ is dense in $\A({U })$  if, for every $A_1,A_2\in \A(U)$ with $A_1\Subset A_2$ there exists $D\in \mathcal D$ such that $A_1\Subset D\Subset A_2$ (for the existence of such a family,  see, e.g.,  \cite[Example 14.8]{DM}).
We adopt the notation $Q_r(x)$ for the open cube centered in $x$ and with its sides parallel to the Cartesian axes and of length $r$. We introduce the following notation for cell minimum problems. Given an open set $U\subseteq \R^d$ and  a   functional $H\colon L^1({U },\R^N)\times\A({U })\to [0,+\infty)$, for every $u\in L^1(U ,\R^N)$ and $V\in \mathcal A(U)$ we set
\begin{equation}\label{eq: minimointornobordo}
   m_H(u,V)=\inf \{H(v,V)\mid v\in L^1(U ,\R^N), v=u \text{ in a neighborhood of } \partial V\}.
\end{equation}
We denote by ${\bf M}^{N\times d} $ and  by ${\bf M}_{\Q}^{N\times d}$  the set of the $N\times d$ matrices whose entries are real  and  rational numbers, respectively. The symbol $|\cdot|$ stays in this context for the Frobenius norm, while occasionally $\Vert\cdot\Vert$ will denote the operator norm.

\begin{definition}[Measure-preserving group action]
A measure-preserving additive group action on a probability space $(\Omega, \mathcal{F}, \mathbb{P})$ is a family $
\tau := \{\tau_z\}_{z \in \mathbb{Z}^d}$ of measurable maps $\tau_z : \Omega \to \Omega$ satisfying the following properties:
\begin{enumerate}
    \item the map $(\omega, z) \mapsto \tau_z(\omega)$ is $(\mathcal{F} \otimes \mathcal{L}^d, \mathcal{F})$-measurable;
    
    \item $\mathbb{P}(\tau_z E) = \mathbb{P}(E)$, for every $E \in \mathcal{F}$ and every $z \in \mathbb{Z}^d$;
    
    \item $\tau_0 = \mathrm{id}_{\Omega}$ and $\tau_{z_1+z_2} = \tau_{z_2} \circ \tau_{z_1}
    \quad \text{for every } z_1, z_2 \in \mathbb{Z}^d.$
\end{enumerate}

If, additionally, every $\tau$-invariant set $E$ has either probability $0$ or $1$, then
$\tau$ is called an \emph{ergodic} measure-preserving group action.
\end{definition}

Note that, in the previous definition, the group acting on $\Omega$ is $\Z^d$, whereas, for other authors, in the definition of measure-preserving group action it is $\R^d$ that acts on the probability space.
\begin{definition}[Subadditive process]\label{defsubproc}
Let $(\Omega, \mathcal{F}, \mathbb{P})$ be a probability space and let  $\tau$ be a measure-preserving group-action on it.  Given a bounded Lipschitz open set
 $U\subseteq \mathbb R^d$, we say that  $\mu : \Omega \times \mathcal{A}_0(U) \to [0, +\infty)$ is a (discrete) subadditive process if it  satisfies 
the following properties:

\begin{enumerate}
    \item \emph{(integrability)} for every $V \in \mathcal{A}_0(U)$, $\mu(\cdot, V)$ belongs to $L^1(\Omega)$;
    
    \item \emph{(stationarity)} for every $\omega \in \Omega$, $V \in \mathcal{A}_0(U)$, and $z \in \mathbb{Z}^d$, it holds
    \[
    \mu(\omega, V + z) = \mu(\tau_z \omega, V);
    \]
    
    \item \emph{(subadditivity)} for every $\omega \in \Omega$, $V \in \mathcal{A}_0(U)$,  for every \emph{finite} family $(V_i)_{i \in I} \subset \mathcal{A}_0(U)$ of pairwise disjoint sets such that $V_i \subset V$ for every $i \in I$ and 
    \[
    \left| V \setminus \bigcup_{i \in I} V_i \right| = 0,
    \]
    there holds
    \[
    \mu(\omega, V) \le \sum_{i \in I} \mu(\omega, V_i).
    \]
\end{enumerate}

If $\tau$ is ergodic (that is, the probability of the $\tau$-invariant sets is $0$ or $1$),  then $\mu$ is called a \emph{subadditive ergodic process}.
\end{definition}

\section{N-Functions and Anisotropic Musielak-Orlicz Spaces}\label{secmusielak}
In this section, we introduce the definitions and statements about the theory of Musielak-Orlicz spaces that are required later in our exposition. Here we follow \cite{Chlebicka}.
\begin{definition}
    We say that $m\colon [0,+\infty)\to [0,+\infty)$ is a \emph{Young function} if it satisfies:
    \begin{enumerate}
\item $m(s) = 0 \Longleftrightarrow s = 0;$
\item $m$ is convex;
\item $m$ is superlinear at zero and at infinity, i.e.
\[
\lim_{s \to 0^+} \frac{m(s)}{s} = 0 \quad \text{and} \quad \lim_{s \to +\infty} \frac{m(s)}{s} = \infty.
\]
\end{enumerate}
\end{definition}

Even though most texts call such a mapping an $N$-function, we reserve this name for a $x$-dependent and anisotropic function.
 In the following definitions,  ${X } \subset \mathbb{R}^d$ will be a measurable set.

\medskip

\begin{definition}[$N$-function]\label{def:N-fun} A function $M\colon {X} \times {\bf M}^{N\times d} \to [0, \infty)$ is called an \emph{$N$-function} if it satisfies the following conditions:
\begin{enumerate}
\item $M$ is a Carathéodory function (i.e.\ measurable with respect to $x$ and continuous with respect to the second variable);
\item $M(x,0) = 0$ and $\Sigma \mapsto M(x,\Sigma)$ is a convex function for a.e.\ $x \in {X}$;
\item (even in $\Sigma$) $M(x,\Sigma) = M(x,-\Sigma)$ for a.e.\ $x \in {X }$ and all $\Sigma \in {\bf M}^{N\times d}$;
\item (boundedness by Young functions) there exist two Young functions $m_1, m_2 : [0,\infty) \to [0,\infty)$ such that for a.e.\ $x \in {X }$ and all $\Sigma \in {\bf M}^{N\times d}$
\begin{equation}\label{m1m2}
m_1(|\Sigma|) \le M(x,\Sigma) \le m_2(|\Sigma|).
\end{equation}
\end{enumerate}
\end{definition}

\begin{rem}
By equivalence of all norms in finite dimensional spaces, in formula \eqref{m1m2} we can actually require that the functions $m_1$ and $m_2$ depend on any matrix norm $\|\cdot\|$.
\end{rem}

We also remark that an $N$-function $M\colon {X} \times {\bf M}^{N\times d} \to [0,\infty)$ could be inhomogeneous and anisotropic, where
\begin{itemize}
\item \emph{inhomogeneity} means dependence on the spatial variable $x \in {X}$,
\item \emph{anisotropy} means dependence on $\Sigma\in {\bf M}^{N\times d}$, not necessarily via $|\Sigma|$.
\end{itemize}
When dealing with $x$-independent, possibly anisotropic functions $M\colon  {\bf M}^{N\times d} \to [0,\infty)$ satisfying the conditions in Definition \ref{def:N-fun} we will keep the name $N$-function to keep a unified exposition.

\begin{definition}[$\Delta_2$-condition]
We say that an $N$-function $M\colon {X } \times {\bf M}^{N\times d} \to [0,\infty)$ satisfies the \emph{$\Delta_2$-condition} (denoted $M \in \Delta_2$) if there exists a constant $c_{\Delta_2} > 0$ and a nonnegative integrable function $h\colon {X} \to \mathbb{R}$ such that
\[
M(x, 2\Sigma) \le c_{\Delta_2} M(x, \Sigma) + h(x) \qquad \text{for a.e. } x \in {X } \text{ and all } \Sigma \in {\bf M}^{N\times d}.
\]
\end{definition}

\begin{definition}[$\nabla_2$-condition]
We say that an $N$-function $M\colon {X } \times {\bf M}^{N\times d} \to [0,\infty)$ satisfies the \emph{$\nabla_2$-condition} (and we write  $M \in \nabla_2$) if there exists a constant $l_{\nabla_2} > 0$ and a nonnegative integrable function $h\colon {X } \to \mathbb{R}$ such that
\[
2 l_{\nabla_2} M(x, \Sigma) \le M(x, l_{\nabla_2}\Sigma) + h(x) \qquad \text{for a.e. } x \in {X } \text{ and all } \Sigma \in {\bf M}^{N\times d}.
\]
\end{definition}

\begin{definition}[Modular] 
By a \emph{modular} we mean a functional $\rho_M$ defined on the set of measurable functions $\xi \colon {X } \to {\bf M}^{N\times d}$ given by the following formula
\[
\rho_M(\xi) := \int_{X } M(x, \xi(x)) \, \mathrm{d}x,
\]
where $M\colon {X } \times {\bf M}^{N\times d} \to [0,\infty)$  is an $N$-function.
\end{definition}

\begin{definition}[Musielak-Orlicz Spaces]
    Let  $M\colon {X }\times {\bf M}^{N\times d}\to [0,+\infty)$ be an $N$-function. We define  \begin{enumerate}
        \item the \emph{generalized Musielak-Orlicz class} $\mathcal L_M({X },{\bf M}^{N\times d})$ as the set of all measurable functions $\xi\colon {X }\to {\bf M}^{N\times d}$ such that $\rho_M(\xi)<\infty$;
        \item the \emph{generalized Musielak-Orlicz space} $L_M({X },{\bf M}^{N\times d})$ as the smallest linear space containing $\mathcal L_M({X },{\bf M}^{N\times d})$.
    \end{enumerate}
\end{definition}

\begin{rem}\label{Lpmusielak}
    For each $p\in (1,\infty)$, the Musielak-Orlicz classes and spaces respectively defined by the $N$-functions $M^p(x,\Sigma)=\frac{|\Sigma|^p}p$ and $M^{\infty}(x,\Sigma)=\infty\chi_{(1,\infty)}(|\Sigma|)$ are the well-known spaces $L^p({X },{\bf M}^{N\times d})$ and $L^{\infty}({X },{\bf M}^{N\times d})$.
\end{rem}

\begin{definition}[Luxemburg norm]
Let  $M\colon {X }\times {\bf M}^{N\times d}\to [0,+\infty)$ be an N-function. Then the mapping $
\|\cdot\|_{L_M} : L_M({X }; {\bf M}^{N\times d}) \to [0,\infty)
$
given by
\[
\|\xi\|_{L_M} := \inf \left\{ \lambda > 0 : \int_{X} M\!\left(x, \frac{\xi(x)}{\lambda}\right) \, \mathrm{d}x \le 1 \right\}
\]
defines a norm, as proven in \cite[Lemma 3.1.9]{Chlebicka}. We call it the \emph{Luxemburg norm}.
\end{definition}

\begin{rem}
Observe that the Luxemburg norm is the so-called Minkowski functional generated by a convex, absorbing and balanced set
\[
A = \left\{ \xi : {X } \to {\bf M}^{N\times d} \ \text{measurable} \ : \ \int_{X } M(x, \xi(x)) \, \mathrm{d}x \le 1 \right\}.
\]
\end{rem}

\begin{rem}\label{embeddingMusielak}
We explicitly note that, if $M\colon {X }\times {\bf M}^{N\times d}\to [0,+\infty)$ is an   $N$-function and $X$ has finite measure,  thanks to  superlinearity of the function $m_1$ in \eqref{m1m2},  then the  continuous embedding   $L_{M}({X };{\bf M}^{N\times d}) \subset L^{1}({X };{\bf M}^{N\times d})$ holds.
\end{rem}

We now introduce an anisotropic version of the $\Delta_2$ property which will be crucial to the density results we are going to prove and use throughout the paper.

\begin{definition}\label{strongdelta2} We say that a function $\gamma\colon{\bf M}^{N\times d}\to [0,+\infty)$ satisfies the \emph{reinforced-$\Delta_2$ property}
 if there exists a function $C\colon [0,+\infty)\to [0,+\infty)$ with superlinear growth near 0, that is \[
 \lim_{s\to 0^+} \frac{C(s)} {s} =0,
 \] such that for every $A \in{\bf M}^{N\times N}$ and $\Sigma\in {\bf M}^{N\times d}$ it holds
 \begin{equation}\label{delta2}\gamma(A \Sigma)\leq C(\|A\|)\gamma(\Sigma),\end{equation}
with $\|A\|$ denoting the operator norm.
\end{definition}

\begin{lemma}
    If  an $N$-function  $\gamma\colon{\bf M}^{N\times d}\to [0,+\infty)$   satisfies the reinforced-$\Delta_2$ property, then $\gamma$ satisfies the $\Delta_2$ and $\nabla_2$ properties.
\end{lemma}
\begin{proof}
    The reinforced-$\Delta_2$ property trivially implies the $\Delta_2$ property. As for $\nabla_2$, denote by $I_N$ the $N\times N$ identity matrix. Fix $\Sigma\in {\bf M}^{N\times d}$. Having chosen $l>0$ such that $2lC(l^{-1})\le 1$ (thanks to the superlinear growth of $C$ near 0), it holds \begin{equation*}
        2l\gamma(\Sigma)=2l\gamma((l^{-1}I_N)(lI_N)\Sigma)\le 2lC(l^{-1})\gamma(l\Sigma)\le \gamma(l\Sigma).
    \end{equation*} \end{proof}

In Definition \ref{strongdelta2} we required $C$ to have superlinear growth near 0 since otherwise the conclusions of the previous lemma would be false, even in the case $d=N=1$. Indeed, in this case, there exist $N$-functions satisfying  property \eqref{delta2}  but not the $\nabla_2$-condition (see, for instance, the counterexample at the beginning of page 34 in \cite{Chlebicka}).

A fundamental tool for our results will be the following approximation result of convex bulk integrands depending on the gradient with smooth functions. Below this result will be proved for bounded functions, with no other requirements than homogeneity and convexity, and later extended to a general Sobolev function if $\gamma$ satisfies the reinforced-$\Delta_2$ condition.
\begin{thm}\label{approximgamma}
Let $\gamma\colon{\bf M}^{N\times d}\to [0,+\infty)$ be a convex  function. 
Let ${U }\subseteq \mathbb R^N$ be a bounded Lipschitz open set and let $u\in W^{1,1}({U },\R^N)\cap L^{\infty}({U },\R^N)$ be such that \begin{equation}\label{finitogamma}\int_{{U }}\gamma (D u)\, \mathrm{d}x<+\infty.\end{equation}Then,  there exists a sequence $\{u_n\}\subseteq   C^{\infty}(\overline{U },\R^N)$ such that, as $n\to +\infty$,  $u_n\to u$ in ${W^{1,1}({U },\R^N)}$ and 
\begin{equation*}\label{approx-gamma}\lim_{n\to +\infty}\int_{{U }}\gamma (D u_n)\, \mathrm{d}x= \int_{{U }}\gamma (D u)\, \mathrm{d}x.\end{equation*}
\end{thm}
\begin{proof}
Using, for instance,  \cite[Proposition 2.5.4]{carbone},  $\overline{U }$ can be covered by a finite collection of open sets $U_j$ such that ${U }\cap {U }_j$ are strongly star-shaped with Lipschitz boundary, that is $\overline{U }\subseteq \bigcup_{i=1}^L {U }_j.$
Let $\alpha_n>0$ be such that $\alpha_n \to 1^+$. Then, as proven in \cite[Lemma 3.6]{muller},  for every $j$ there exists a sequence $z_n^j\in C^\infty (\overline{{U }\cap {U }_j},\R^N)$ such that $z_n^j\to u$ in $W^{1,1}({U }\cap {U }_j,\R^N)$,  $$\lVert z_n^j \rVert_{L^\infty({U }\cap {U }_j,\R^N)}\le \alpha_n\lVert u \rVert_{L^\infty({U }\cap {U }_j,\R^N)}\quad \hbox{ for every $j=1,\dots, L$,\   $n\in\N$},$$ and\footnote{the statement  of \cite[Lemma 3.6]{muller} only mentions convergence of the integrals, but the proof provides indeed $L^1$-convergence (see \cite[Page 206]{muller}).}
\begin{equation*}
\lim_{n \to +\infty} \|\gamma(Dz_n^j)-\gamma(Du)\|_{L^1({U \cap U_j })}=0.
\end{equation*}
Up to replacing $z_n^j$ with a subsequence,  we can assume that  $z_n^j\to u$ and $Dz_n^j\to Du$ a.e.  in ${U \cap U_j }$  for every $j$. 
Now, given a smooth partition of unity $\{\eta_j\}$ subordinate to the covering $\{{U }_j\}$, for every $n\in\N$ we define $u_n=\sum_{j=1}^L \eta_j z_n^j$. Obviously $u_n\in C^{\infty}(\overline{U },\R^N)$. Notice that, as $\eta_j$ are compactly supported in $U_j$, by exploiting the above convergences in $L^1$,  it holds
\begin{equation}\label{partizione}
\lim_{n\to +\infty} \int_{{U }}\eta_j \gamma(Dz_n^j)\,\mathrm{d}x=\int_{{U }}\eta_j\gamma(Du)\,\mathrm{d}x
\end{equation} 
for each $j=1, \dots L$. Moreover, being  $u_n$ a convex combination of the functions $z_n^1, z_n^2,\cdots, z_n^L$, we have that $u_n\to u$ in $L^1({{U },\R^N})$.
As for the gradients, since $\sum_{j=1}^L D\eta_j=0$, we have \[Du_n-Du=\sum_{j=1}^L \eta_j (Dz_n^j-Du)+\sum_{j=1}^L D\eta_j \otimes (z_n^j-u)
    \]
and, thanks to the uniform boundedness of $D \eta_j$, we have $Du_n\to Du$ in $L^1({{U },\R^N})$. 
Now, for every fixed $t\in (0,1)$,  using  the   convexity of $\gamma$, we   obtain that
\begin{equation}\label{primoint}
\begin{split}
& \limsup_{n\to +\infty} \int_{U} \gamma(tDu_n)\, \mathrm{d}x\\
&\le t \limsup_{n\to +\infty} \sum_{j=1}^L \int_{U} \eta_j \gamma(Dz_n^j)\,\mathrm{d}x
+ (1-t) \limsup_{n\to +\infty} \int_{U} \gamma\Bigg(\sum_{j=1}^L \frac{t}{1-t} D\eta_j \otimes (z_n^j-u)\Bigg)\,\mathrm{d}x.\\
\end{split}
  \end{equation}
Thanks to the boundedness of $u$ and to the fact that 
\[
\lVert z_n^j \rVert_{L^\infty(U\cap U_j,\mathbb{R}^N)} \le \alpha_n \lVert u \rVert_{L^\infty(U,\mathbb{R}^N)}
\]
for every $j=1,\dots, L$ and for every $n\in \mathbb{N}$, we find a constant $c=c(t,\|u\|_{\infty})$ such that, for every $n\in \mathbb{N}$ and for a.e. $x\in U$, it holds
\[
\gamma\Bigg(\sum_{j=1}^L \frac{t}{1-t} D \eta_j \otimes (z_n^j-u)\Bigg)\le \sup_{\|\Sigma\|\le c}\gamma(\Sigma)<+\infty,
\]
since $\gamma$ is locally bounded by continuity. Hence, for every fixed $t\in (0,1)$, the Dominated Convergence Theorem yields  
\begin{equation}\label{secondoint}        
\lim_{n \to +\infty} \int_U \gamma\Bigg(\frac{t}{1-t} \sum_{j=1}^L D \eta_j \otimes (z_n^j-u)\Bigg)\,\mathrm{d}x = |U|\gamma(O).
\end{equation}
By combining \eqref{primoint}, \eqref{partizione} and \eqref{secondoint}, and since $\sum_{j=1}^L \eta_j=1$, we get
\begin{equation}\label{eq:inequality}
\limsup_{n\to +\infty} \int_U \gamma(t Du_n)\, \mathrm{d}x \leq 
t \int_U \gamma(Du)\, \mathrm{d}x + (1-t)\, |U|\gamma(O)
\qquad \text{for every } t\in (0,1).
\end{equation}

Now, fix a sequence $t_k$ such that $t_k\to 1^-$ as $k \to +\infty$. Thanks to \eqref{eq:inequality}, for each fixed $k$ we can iteratively find $n_k>n_{k-1}$ such that
\[
\int_U \gamma(t_k Du_{n_k})\, \mathrm{d}x
\leq \limsup_{n\to +\infty}\int_U \gamma(t_k Du_n)\,\mathrm{d}x + \frac{1}{2^k}
\leq t_k \int_U \gamma(Du)\,\mathrm{d}x + (1-t_k)\, |U|\gamma(O) + \frac{1}{2^k}.
\]
If we now set $v_k= t_k u_{n_k}$, we have that $v_k \to u$ in $W^{1,1}(U, \mathbb{R}^N)$ and, by the previous inequality,
\[
\limsup_{k\to +\infty} \int_U \gamma(Dv_k)\, \mathrm{d}x
\leq \int_U \gamma(Du)\, \mathrm{d}x.
\]
On the other hand, by Fatou's Lemma we have
\[
\liminf_{k\to +\infty} \int_U \gamma (Dv_k)\, \mathrm{d}x 
\ge \int_U \gamma (D u)\, \mathrm{d}x,
\]
which, combined with the reverse inequality above, gives the desired conclusion. \end{proof}

The following Lemma introduces a general truncation procedure which will be used several times in the paper.

\begin{lemma}\label{tronc}
Let $U$ be an open set and $u\in W^{1,1}({U },\R^N)$. Then, for every $M\in \mathbb N$, $M\geq 2$,   there exists a function $u^M\in W^{1,1}({U },\R^N)\cap L^\infty({U },\R^N) $ with $\|u^M\|_{L^\infty(U,\R^N)}\leq 2M$,  such that, for a.e. $x\in {U }$,  it holds
\begin{enumerate}
\item [$(i)$] $u^M(x)=u(x) $ if $|u(x)|\leq M$; 
 \item [$(ii)$] $u^M(x)=0 $ if $|u(x)|\geq 2M$;
\item [$(iii)$] $|u^M(x)-v^M(x)|\leq |u(x)-v(x)| $ for every $u,v\in W^{1,1}({U },\R^N)$;
\item [$(iv)$] $u^M(x)\to u(x)$  as $M\to +\infty$;
\item [$(v)$]  $|Du^M(x)|\leq |Du(x)|$ and, as $M\to +\infty$,   $|Du^M(x)|\to  |Du(x)|$.
\end{enumerate}
In particular,  as $M\to +\infty$, $u^M\to u$ in $W^{1,1}({U },\R^N)$.  
\end{lemma}
\begin{proof} Let us consider the   function $\phi_M\in C^1([0,+\infty))$ given by 
\[\phi_M(s)=\begin{cases}1, \qquad \hbox{ if } 0\leq s\leq M,\\
\\
2\left(\frac s M-1\right)^3-3\left(\frac {s} M-1\right)^2+1, \qquad \hbox{ if } M\leq s\leq 2M,\\
\\
0, \qquad \hbox{ if } s \geq 2M,
\end{cases}
\]
and we consider  the function $\psi_M: \R^N\to \R^N$ be given by  $$\psi_M(\xi)=\phi_M(|\xi|)\xi.$$
Then $\psi_M\in  C_0^1(\R^N,\R^N)$, $|\psi_M(\xi)|\leq 2M$ for every $\xi\in \R^N$, $\psi_M(\xi)=0$ if $|\xi|\geq 2M$. Furthermore,  $\psi_M$ is $1$-Lipschitz (see, for instance, \cite[Formula~(4.5)]{CDSZ}) so that, in particular, the operator norm of the gradient satisfies $\|D \psi_M(\xi)\|\leq 1$ for every $\xi\in \R^N$. For every $u\colon{U }\to \mathbb R^N$, we define
\begin{equation}\label{troncata}  u^M(x)=\psi_M(u(x)).
\end{equation}
Then it is easy to show that $u^M$ satisfies $(i), (ii), (iii), (iv)$.
Furthermore (see, e.g., \cite[Exercise 10.38]{Leoni}) $u^M\in W^{1,1}({U },\R^N)\cap L^\infty({U },\R^N) $ and, for a.e. $x\in {U }$, the following chain rule   holds
 \begin{equation}\label{chain} D(u^M)(x)=D\psi_M(u(x)) D u (x).\end{equation}
 Then \begin{equation}\label{monotonia}
|D(u^M)(x)|=|D\psi_M(u(x)) D u (x)|\leq  \|D\psi_M(u(x))\| |D u (x)| \leq | D u(x) | .
\end{equation}
Moreover,  if $x\in {U }$ is a Lebesgue point for $u$ and for $Du$,  and  $ |u(x)|\leq M$,  then $D(u^M)(x)=D u (x).$ 
In particular  we have that  $D(u^M)(x)\to D u (x) $ as $M\to +\infty.$
Now $$\|u^M-u\|_{L^1({U },\R^N)}=\int_{M\leq \|u\|\leq 2M } \phi_M(|u(x)|)|u(x)|\,\mathrm{d}x\, +  \int_{|u(x)|\geq 2M } |u(x)|\,\mathrm{d}x\, \leq \int_{ \|u\|\geq M } |u(x)|\,\mathrm{d}x $$
which tends to $0$ as $M\to +\infty$.
Finally, by using \eqref{monotonia} and the Dominated convergence theorem, we get that $\|Du^M-Du\|_{L^1({U },\R^N)}\to 0$ as $M\to +\infty$. This concludes the proof.
\end{proof}

With the previous Lemma and the reinforced-$\Delta_2$ property we can extend Theorem \ref{approximgamma} to a general $u\in W^{1,1}({U },\R^N)$ when $U$ is a bounded Lipschitz open set.

\begin{cor} \label{senzaLinfty} Let ${U }\subseteq \R^d$ be a bounded Lipschitz open set and let $\gamma\colon{\bf M}^{N\times d}\to [0,+\infty)$ be a convex  function satisfying  the reinforced-$\Delta_2$ property. Then the conclusions of  Theorem \ref{approximgamma} holds for every  $u\in W^{1,1}({U },\R^N)$ satisfying \eqref{finitogamma}.

\end{cor}
\begin{proof}
For every $v\in W^{1,1}({U },\R^N)$ we define \[G(v)=\int_{{U }}\gamma(Dv)\,\mathrm{d}x.\]
Let  $u\in W^{1,1}({U },\R^N)$ be such that $G(u)<+\infty$ and 
let $u^M$ be defined by \eqref{troncata}.  Then, since $\gamma$ satisfies the reinforced-$\Delta_2$ property and $D\psi_M(\xi)= 1$ for every $|\xi|\leq M$, the chain rule \eqref{chain} implies that 
\begin{align*}
\int_{U } \gamma(Du^M)\, \mathrm{d}x&=  \int_ { \{|u|<  M\} }\gamma(Du)\, \mathrm{d}x + \int_{ \{|u|\geq  M\} }\gamma(D\psi_M(u(x)) D u (x))\, \mathrm{d}x\\
&\le \int_{\{|u|<  M\} }\gamma(Du)\, \mathrm{d}x + C(1)\int_{\{|u|\geq  M\}}\gamma(Du)\, \mathrm{d}x
\end{align*}
    and the last integral tends to 0 as $M\to +\infty$ since $\gamma(Du)\in L^1({U })$. As a result, \begin{equation}\label{dallalto}
        \limsup_{M\to +\infty}\int_{{U }} \gamma(D u^M)\, \mathrm{d}x\le \int_{{U }} \gamma(D u)\, \mathrm{d}x.
    \end{equation}
    Since $Du^M\to Du$ a.e. $x\in {U }$
 as $M\to +\infty$,      
  by using Fatou's Lemma, we obtain     \begin{equation}\label{dalbasso}
    \liminf_{M\to +\infty}\int_{{U }} \gamma(Du^M)\, \mathrm{d}x\ge \int_{{U }} \gamma(D u)\, \mathrm{d}x.
    \end{equation}
Combining \eqref{dalbasso} and \eqref{dallalto}, we find that 
$$G(u)= \lim_{M\to +\infty} G(u^M).$$
With this, the result follows now by a standard double sequence argument.
\end{proof}
The result of Corollary \ref{senzaLinfty} can be also restated in terms of piecewise affine functions.
\begin{definition} Let ${U }\subseteq \R^d$ be an open set. We say that a function $v\in W^{1,\infty}({U },\R^N)$ is a piecewise affine function in ${U }$ if there exists  $\{{U }_k\}_{k\in \N}\subset \mathcal A({U })$ such that
\begin{itemize}
\item ${U }_k\cap {U }_j=\emptyset$ for every $k,j\in \N$ such that $k\not=j$;
\item $|{U }\setminus \bigcup_{k\in \N} {U }_k|=0;$
\item $v$ is affine on each ${U }_k$, i.e. there exists $\Sigma_k\in {\bf M}^{N\times d}$ and $q_k\in \R$ such that $$v(x)=\Sigma_k x+q_k, \qquad \forall\, x\in {U }_k, \ k\in \N.$$
\end{itemize}
\end{definition}
\begin{cor}\label{PAdensity}
Let $\gamma\colon{\bf M}^{N\times d}\to  [0,+\infty)$ be a convex function  satisfying  the reinforced-$\Delta_2$ property. 
Let ${U }\subseteq \R^d$ be a bounded Lipschitz open set and let $u\in W^{1,1}({U },\R^N)$. If $u$ satisfies \eqref{finitogamma}, then there exists a sequence $\{u_n\}$ of piecewise affine functions such that, as $n\to +\infty$,  $u_n\to u$ in $W^{1,1}({U },\R^N)$ and \begin{equation*}\label{approx-gamma}\lim_{n\to +\infty}\int_{{U }}\gamma (D u_n)\, \mathrm{d}x= \int_{{U }}\gamma (D u)\, \mathrm{d}x.\end{equation*}
\end{cor}
\begin{proof} 
As piecewise affine functions are dense in $C^1(\overline{U})$ with respect to the $W^{1, \infty}$ norm and $\gamma$ is a continuous function by convexity, the result follows again by Corollary \ref{senzaLinfty} and a double sequence argument. \end{proof}

We finally endow the set of $W^{1,1}$ functions satisfying \eqref{finitogamma} with a Banach space structure.
\begin{thm}\label{musielakbanach}  Let $\gamma\colon {\bf M}^{N\times d}\to [0,+\infty)$ be an $N$-function satisfying $\Delta_2$ and let $U\subseteq \R^d$ be a bounded open set. Then $$W_{\gamma}({U },\R^N)=\{ u\in L^1({U },\R^N):\, D u\in  L_{\gamma}({U },{\bf M}^{N\times d})\},$$
endowed with the norm $$\|u\|_{W_{\gamma}}=\|u\|_{L^1}+\|Du\|_{L_{\gamma}},$$ 
  is a separable Banach space. Moreover, if ${U }$ is a  Lipschitz open set and  $\gamma$ satisfies both properties $\Delta_2$ and  $\nabla_2$,  then $W_{\gamma}({U },\R^N)$ is reflexive and compactly embedded in $L^1({U },\R^N)$.
\end{thm}
    \begin{proof}
        \cite[Theorem 3.1.18]{Chlebicka}   guarantees that $L_{\gamma}({U },{\bf M}^{N\times d})$, which is the smallest linear space containing $\mathcal L_{\gamma}({U },{\bf M}^{N\times d})$, is a  Banach space when endowed with the  Luxemburg norm. 
         \cite[Lemma 3.3.1]{Chlebicka} ensures that, in view of the $\Delta_2$ condition, these two spaces coincide. 
      Thanks to    \cite[Corollary 3.4.15, Lemma 3.1.19]{Chlebicka}   we have that   $L_{\gamma}({U },{\bf M}^{N\times d})$  is  separable  when endowed with the  Luxemburg norm.  Hence, since  $W_{\gamma}({U },\R^N)$ is closed with respect to the 
     strong convergence in    $L^1({U },\R^N)\times  L_{\gamma}({U },{\bf M}^{N\times d})$,  it is a separable Banach space when endowed with the norm  $\|\cdot\|_{W_{\gamma}}$. Finally, \cite[Remark 3.3.3]{Chlebicka} states that, also due to $\nabla_2$, $L_{\gamma}({U },{\bf M}^{N\times d})$ is reflexive. Then, since  $W_{\gamma}({U },\R^N)$ is continuously embedded in $W^{1,1}({U },\R^N)$ (see Remark \ref{embeddingMusielak}),  by the compact embedding of $W^{1,1}({U },\R^N)$ in $L^1({U },\R^N)$, we have that any bounded sequence in $W_{\gamma}({U },\R^N)$ is weakly compact in $L^1({U },\R^N)\times  L_{\gamma}({U },{\bf M}^{N\times d})$,
     hence also in $W_{\gamma}({U },\R^N)$.
\end{proof}

\section{A model case in convex homogenization}\label{convex homog}
In this section, we focus on a model case in periodic convex homogenization with non-standard growth. We consider a Carath\'eodory integrand $g\colon \R^d \times \R \to [0, +\infty)$ and $1<p^-\le p^+<+\infty$ satisfying the following assumptions:
\begin{enumerate}
\item for a.e.\ $x \in \mathbb{R}^d$  $g(x,\cdot)$ is convex and increasing  on $[0,+\infty)$;

\item $g(\cdot,s)$ is $1$-periodic for every $s \in \mathbb{R}$;
 
\item  {\normalfont(inc)$_{p^-}$}
\; the map
$s \in (0,+\infty) \mapsto \dfrac{g(x,s)}{s^{p^-}}$
is increasing for a.e.\ $x \in \mathbb{R}^d$;

\item {\normalfont(dec)$_{p^+}$}
\; the map
$s \in (0,+\infty) \mapsto \dfrac{g(x,s)}{s^{p^+}}$
is decreasing for a.e.\ $x \in \mathbb{R}^d$;

\item {\normalfont(A0)}
\; there exists $\alpha, \beta > 0$ such that
\[
\alpha \le g(x,1) \le \beta
\]
for every  a.e.\ $x \in \mathbb{R}^d$.

\end{enumerate}
It is well-known (see for instance \cite[Section 3]{HastoOk}) that, under \normalfont(inc)$_{p^-}$ and \normalfont(dec)$_{p^+},$ we have 
\begin{equation}\label{cons2}
\min\{s^{p^-}, s^{p^+}\}g(x,t)\le g(x,st)\le\max\{s^{p^-}, s^{p^+}\}g(x,t),
\end{equation}
 for a.e.  $x\in{\R^d}$, for every $ t\in \R$ and $s\geq 0$. By \normalfont(A0), taking $t=1$ we deduce
\begin{equation}\label{growth-p-p+}
\alpha \min\{s^{p^-}, s^{p^+}\} \le g(x,s) \le \beta \max\{s^{p^-}, s^{p^+}\}
\end{equation}
so that in particular $g(x,0)=0$ and for a.e.  $x\in{\R^d}$ and for every $s \geq 0$  it holds
    \begin{equation}\label{growth-g}\alpha(s^{p^-}-1)\le g(x,s) \le \beta\left(1 +s^{p^+}\right).
    \end{equation}
With a slight abuse of notation we define $g(x, \Sigma):=g(x, |\Sigma|)$ for all $\Sigma \in {\bf M}^{N\times d}$ and accordingly we consider the functional $G_{\varepsilon}:L^1({U };\mathbb R^N)\times \mathcal{A}({U })\to [0,+\infty]$ defined by
\begin{equation}\label{casovett}
G_{\varepsilon}(u, V) = 
\begin{cases} 
\displaystyle \int_V g\left(\frac{x}{\varepsilon}, Du\right) \, \mathrm{d}x & \text{if } u \in W^{1,1}(V; \R^N), \\
+\infty & \text{otherwise,}
\end{cases}
\end{equation}
and set $G_{\varepsilon}(u):=G_{\varepsilon}(u, {U })$. Notice that the integral is always well-defined if $u \in W^{1,1}(V; \R^N)$, but could possibly take the value $+\infty$, while it is for sure finite if  $Du \in L^{p^+}(V; {\bf M}^{N\times d})$.
A reference case to keep in mind is the variable exponent $g(x, Du)=|Du|^{p(x)}$ where $p\colon{\R^d }\to\R$ is a fixed  $1$-periodic measurable function, with periodicity cell $(0,1)^d$,  satisfying   
\[
p^-\le p(x)\le p^+ \hbox{ for a.e. } x\in{\R^d }\,.
\]

This section is devoted to prove the following theorem. In the statement, besides the $\Gamma$-convergence result on ${U }$, we  remark that, as a byproduct of the proof, we get that the localized $\Gamma$-liminf and $\Gamma$-limsup indeed agree on each open subset of ${U }$ provided that $\gamma(Du)$ has a finite integral on the {\it whole} reference set ${U }$.  This  will be exploited in the localization procedure  in the next section.
We recall from Section \ref{notaz} that, given a sequence ${\varepsilon_k}\to 0$, for every $V\in \mathcal A({U })$ we 
 define  \begin{align*}
    G'(u,V)=\Gamma(L^1)\text{-}\liminf_{k\to +\infty} G_{\varepsilon_k}(u,V), \qquad 
G''(u,V)=\Gamma(L^1)\text{-}\limsup_{k\to +\infty} G_{\varepsilon_k}(u,V).
\end{align*}
In principle, $G'$ and $G''$ may depend on the chosen subsequence $\varepsilon_k$. In the statement below, we will however prove a representation formula independent of the chosen subsequence.
\begin{thm}\label{gammamodel}  Let  ${U }\subseteq \mathbb R^d$ be a bounded connected Lipschitz open set and let    $G_\varepsilon$  be given by \eqref{casovett}.
Then, there exists an $N$-function $\gamma: {\bf M}^{N\times d}\to [0,+\infty)$
which satisfies  the reinforced-$\Delta_2$ property 
\begin{equation}\label{Delta2gamma}\gamma(A \Sigma)\leq \max\{\|A\|^{p^-},\|A\|^{p^+}\}\gamma(\Sigma) \qquad  \forall A\in {\bf M}^{N\times N},  \forall \Sigma \in {\bf M}^{N\times d}\,,  \end{equation}
such that the family $\{G_\varepsilon\}$
$\Gamma$-converges, with respect to the $L^1({U },\R^N)$ norm, to the functional 
 \begin{align}\label{model-limit}
G(u):=
\left\{
\begin{array}{l}
\displaystyle\int_{{U }} \gamma(D u)\,\mathrm{d}x\quad \mbox{if }u\in W_{\gamma}({U },\R^N), \\
+\infty \quad \mbox{elsewhere in }L^1({U },\R^N).
\end{array}
\right.   
\end{align}
Moreover, $\gamma$ satisfies  the following cell formula:
\begin{align}\label{fhom}
\gamma(\Sigma)=\inf_{W_{{\rm per}}^{1,p^-}({(0,1)^d}, {\bf \R}^{ N})} \int_{{(0,1)^d}} g(y, \Sigma + Du(y)) \, \mathrm{d}y
\end{align}
where $$W_{{\rm per}}^{1,p^-}({(0,1)^d},{\bf \R}^{ N})=\{ u\in W^{1,p^-}_{loc}(\R^d,{\bf \R}^{ N}): \ u \ 1\hbox{-periodic} \}$$
is the closure of smooth $1$-periodic functions in  $W^{1,p^-}_{loc}(\R^d,{\bf \R}^{ N})$.

\noindent Furthermore, if $u\in W_{\gamma}({U }, \R^N)$, then, for every sequence $\varepsilon_k\to 0$ and every $V\in\A({U })$, we have  
\begin{equation}\label{eq: local}
G'(u, V)=G''(u, V)=\int_V\gamma(Du)\,\mathrm{d}x.
\end{equation}
\end{thm}

\begin{rem}
By the growth condition from below and the De Giorgi-Ioffe lower semicontinuity theorem for convex integrands in the gradient variable, the infimum in \eqref{fhom} is indeed a minimum.
 \end{rem}

\begin{rem}\label{no-periodicity}
We  explicitly note that in the next results for the family $\{G_{\varepsilon}\}$, going from Lemma \ref{boundsp+} to Proposition \ref{G''min},  we do not  make  use of the $1$-periodicity of $g$ with respect to $x$.  This  observation will be used to apply all these results  in Section \ref{stochastic}, where periodicity is replaced by a stationarity assumption.
 \end{rem}

The next lemma recalls the fundamental estimate  for the family $G_{\varepsilon}$, which will be used several times in the paper.
\begin{lemma}\label{fundest}
    The family $\{G_{\varepsilon}\}$ satisfies the following uniform fundamental estimate:
    for every $V, V', W \in \mathcal{A}({U })$ with $V' \Subset V$ and $\sigma > 0$ there exists $M_\sigma > 0$ such that for all $u, v \in L^1({U },\mathbb R^N)$ there exists a cut-off function $\varphi$ between $V'$ and $V$ such that for every $\varepsilon>0$ it holds
\begin{equation}\label{eq:lp_fundamental_estimate}
G_{\varepsilon}(\varphi u + (1 - \varphi)v, V' \cup W) \leq (1 + \sigma)(G_{\varepsilon}(u, V) + G_{\varepsilon}(v, W)) + M_\sigma \int_{(V \cap W) \setminus V'} |u - v|^{p^+} \, \mathrm{d}x + \sigma. 
\end{equation}    
\end{lemma}
\begin{proof} It  is sufficient to apply \cite[Proposition 21.7]{BDF} (see also \cite[Theorem 6.1]{DMM}).
\end{proof}

Throughout this section, we initially consider a countable subset $\mathcal D\subseteq \A_0({U })$ such that ${U } \in \mathcal D$, $\mathcal D$ is dense in $\A({U })$  
and stable under finite union,  and  a sequence  ${\varepsilon_k}\to 0$ is fixed  with the following property:  
\begin{equation}\label{epsk}
     G'(\cdot,D)=G''(\cdot, D) \quad \hbox{ for every } D \in \mathcal D.
 \end{equation}
As the considered subset $\mathcal D$ of $\A({U })$ is countable, this fact can be ensured by compactness of the $\Gamma$-convergence (see for instance \cite[Theorem 10.3]{BDF}) and a diagonal argument. The need for taking such a subsequence will be eventually removed in the proof of Theorem \ref{gammamodel}.

\begin{lemma} \label{boundsp+} Let  $V\in \mathcal A({U })$. Then 
\begin{enumerate}
\item    for every $u\in W^{1,p^+}(V;\mathbb R^N)$ it holds
\begin{align}\label{altop+}
        G''(u,V)\le  \int_V	\beta\left(1+|D u|^{p^+}\right)\,\mathrm{d}x.
    \end{align}
 \item Moreover, if $V\in \mathcal A({U })$,  then
for every $u\in L^{1}({U };\mathbb R^N)$ such that $G'(u,V)<+\infty$, we have that $Du\in L^{p^-}(V;\mathbb R^N)$ and  it holds
\begin{align}\label{bassop-}
      \int_V\alpha\left( |D u|^{p^-}-1\right)\,\mathrm{d}x\le G'(u,V).
    \end{align}
\end{enumerate}
\end{lemma}

\begin{proof}
In order to show \eqref{altop+}, it is sufficient to note that,  by definition and   thanks  to the growth condition  \eqref{growth-g}, it holds
\[
\begin{split}
        G''(u,V)&\le \limsup_{k\to +\infty} G_{\varepsilon_k}(u,V)\leq \beta \int_V  \left(1+|D u|^{p^+}\right)\, \mathrm{d}x.
    \end{split}
    \]
In order  to prove \eqref{bassop-}, take a recovery sequence $\{w_{k}\}\subseteq L^1({U };\mathbb R^N)$ for  $G'(u,V)<+\infty$. Then $w_{k}\to  u$ in $L^1(V,\mathbb R^N)$ and there exists $k_0\in \mathbb N$ large enough such that 
 $$ 
\sup_{k\geq k_0}    G_{\varepsilon_k}(w_{k},V)<+\infty.
$$ 
By applying \eqref{growth-g}, it follows that 
 $$\sup_{k\geq k_0}  \int_V  \alpha \left(|D w_{{k}}|^{p^-}-1\right)\,\mathrm{d}x\leq \sup_{k\geq k_0}   \displaystyle \int_V g\left(\frac x{\varepsilon_k}, Dw_{{k}}(x)\right)\, \mathrm{d}x $$
 Hence the sequence $\{Dw_k\}_{k\geq k_0}  $ is bounded in $L^{p^-}(V, {\bf M}^{N\times d})$. This implies that, up to a subsequence, $Dw_k\wto v\in L^{p^-}(V, {\bf M}^{N\times d})$. Since $w_k\to u $ in $L^1(V,\mathbb R^N)$, we get that $Du=v$ on $V$. 
    Therefore  we get that  \begin{align*}
\alpha  \int_V   \left(|D u|^{p^-}-1\right)\,\mathrm{d}x\leq    \alpha \liminf_{k\to +\infty}   \int_V   \left(|D w_k|^{p^-}-1\right)\,\mathrm{d}x\le   \liminf_{k\to +\infty} G_{\varepsilon_k}(w_k,V)=G'(u,V),
    \end{align*}
which gives the desired conclusion. \end{proof}

We now prove a subadditivity property for $G'(u, \cdot)$ and $G''(u, \cdot)$ which relies on a truncation argument together with the fundamental estimate.
\begin{lemma}\label{subadditivity}
  Let  $u\in L^1({U };\mathbb R^N)$. Then $G'(u,\cdot)$ is superadditive on $\mathcal A({U })$.
Moreover,  for every $V,V',W\in \mathcal A({U })$ and  $V'\Subset V$ we have:
    \begin{align}
        G'(u,V'\cup W)\le G'(u,V)+G''(u,W),\label{eq: subadd} \\
        G''(u,V'\cup W)\le G''(u,V)+G''(u,W) \label{eq: subaddG''} .
    \end{align}
\end{lemma}
\begin{proof}
 In order to show that  $G'(u,\cdot)$ is superadditive, let $V,W\in\A({U })$ be such that  $V\cap W=\varnothing$ and  let us  assume that $G'(u, V\cup W)<+\infty$.  Let     $(u_k)$ be a recovery sequence for  $G'(u, V\cup W)$. Hence  we have:
\begin{align*}
G'(u,V\cup W)&=\liminf_{k\to +\infty} G_{{\varepsilon_k}}(u_k,V\cup W)=\liminf_{k\to +\infty} \int_{V\cup W} g\left(\frac x {{\varepsilon_k}}, Du_k\right)\, \mathrm{d}x=\\
&= \liminf_{k\to +\infty} \left( G_{{\varepsilon_k}}(u_k,V)+G_{{\varepsilon_k}}(u_k,W)\right)\\
&\geq \liminf_{k\to +\infty} G_{{\varepsilon_k}}(u_k,V)+ \liminf_{k\to +\infty}  G_{{\varepsilon_k}}(u_k,W)
\geq  G'(u,V)+ G'(u,W)
\end{align*}
which proves the superadditivity of $G'$. 
   Now  we prove \eqref{eq: subadd}; the proof of \eqref{eq: subaddG''} is analogous. Assume that $G'(u,V)$ and $G''(u,W)$ are finite; otherwise the thesis is trivial.  Hence,  
 there exist $\{u_k\} \subseteq L^{1}(V, \mathbb R^N)$ and $ \{v_k\} \subseteq  L^{1}(W, \mathbb R^N)$ such that  
  \[Du_k\in L^{p^-}(V,{\bf M}^{d\times N}), \  u_k\to u \ \hbox{ in } L^1(V, \mathbb R^N) \ \hbox{ and } \ \liminf_{k\to +\infty} G_{\varepsilon_k}(u_k,V)=  G'(u,V),\] 
  \[Dv_k\in L^{p^-}(W,{\bf M}^{d\times N} ), \  v_k\to u \ \hbox{ in } L^1(W,  \mathbb R^N)  \ \hbox{ and } \  \limsup_{k\to +\infty} G_{\varepsilon_k}(v_k,W)=  G''(u,W).\]
Thanks to Lemma \ref{boundsp+}, Part 2, we have that $u\in W^{1,1}({V },\mathbb R^N)\cap W^{1,1}({W },\mathbb R^N) $. Then, for every fixed $M>0$,  we consider the functions $u^M $,  $u_k^M\in W^{1,1}({V },\mathbb R^N)\cap L^\infty({V },\mathbb R^N)$ and  $v_k^M\in W^{1,1}({W },\mathbb R^N)\cap L^\infty({W },\mathbb R^N)$ given by Lemma \ref{tronc}.
Since $|Du_k^M(x)| \leq | D u_k (x) | $ for a.e.\ $x\in {V }$ we obtain that  $Du_k^M\in L^{p^-}({V }, {\bf M}^{N\times d})$.  Similarly, $Dv_k^M\in L^{p^-}({W}, {\bf M}^{N\times d})$.
Now,  thanks to Lemma \ref{fundest}, for every $\sigma>0$,  there exists $M_\sigma > 0$ and 
    a cut-off function $\phi_{\varepsilon_k}$ between $V'$ and $V$ such that,   defined $w^M_k=\phi_{k} u_k^M+(1-\phi_{k})v_k^M$,  one has
\begin{equation}\label{fe}
G_{\varepsilon_k}(w^M_k,V'\cup W)\le (1+\sigma)\bigl[G_{\varepsilon_k}(u_k^M,V)+G_{\varepsilon_k}(v_k^M,W)\bigr]+M_\sigma\int_{V\cap W} |u_k^M-v_k^M|^{p^+}\, \mathrm{d}x + \sigma,
 \end{equation}
 for every $k\in\N$. 
 Thanks to Lemma \ref{tronc}, Part $(iii)$,  we have  that, for every $M\geq 2$, $$\|u_k^M - u^M\|_{L^{p^+}(V,\mathbb R^N)}\leq \|u_k^M - u^M\|^{\frac 1{p^+}}_{L^{1}(V,\mathbb R^N)}(4M)^{\frac{p^+-1}{p^+}}\leq \|u_k - u\|^{\frac 1{p^+}}_{L^{1}(V,\mathbb R^N)}(4M)^{\frac{p^+-1}{p^+}}\  .$$
 Hence, passing to the limit as $k\to +\infty$, the above inequality implies that   $\|u_k^M - u^M\|_{L^{p^+}(V,\mathbb R^N)}\to 0$. 
 Analogously,  as $k\to +\infty$,  we have that  $v_k^M\to u^M$ in  $L^{p^+}(W, \mathbb R^N)$.
Hence,   $w^M_k\to u^M$ in $L^1(V'\cup W, \mathbb R^N)$ as $k\to +\infty$  and 
 \begin{equation}\label{conv} \|u_k^M - v_k^M\|_{L^{p^+}(V\cap W,\mathbb R^N)}\to 0 \hbox{ as } k\to +\infty.
 \end{equation}
By definition of $G'$,  this implies \begin{equation}\label{eq:sx}
    G'(u^M,V'\cup W)\le \liminf_{k\to +\infty} G_{\varepsilon_k}(w^M_k,V'\cup W).
    \end{equation}
    For the right-hand side of \eqref{fe},   in order to bound from above, by using  property \eqref{monotonia} and since $g(x, \cdot)$ is radially increasing,  we get that   \begin{equation}\label{eq:dx}
    G_{\varepsilon_k}(u_k^M,V)\le G_{\varepsilon_k}(u_k,V), \qquad G_{\varepsilon_k}(v_k^M,W)\le G_{\varepsilon_k}(v_k,W).
    \end{equation}
 By passing to   the liminf as $k\to +\infty$ in   \eqref{fe} and by exploiting \eqref{eq:sx},  \eqref{eq:dx} and \eqref{conv},  we deduce:
 \[\begin{split}
         G'(u^M,V'\cup W)&\le (1+\sigma)\liminf_{k\to +\infty} \left(G_{\varepsilon_k}(u_k,V)+G_{\varepsilon_k}(v_k,W)\right)+\sigma\\
         &\le (1+\sigma)\left( \liminf_{k\to +\infty} G_{\varepsilon_k}(u_k,V)+  \limsup_{k\to +\infty} G_{\varepsilon_k}(v_k,W)\right)+\sigma\\
         &=(1+\sigma)\left(G'(u,V)+G''(u,W)\right)+ \sigma.
         \end{split}
\]
    Finally,  sending   $M\to\infty$ and then  $\sigma\to 0$, since  $u^M\to u$ in $L^1({U };\mathbb R^N)$, the $L^1({U })$-lower semicontinuity of $G'$ implies  \eqref{eq: subadd}.
\end{proof}
We deduce from the previous result an inner regularity property for $G'(u, \cdot)$ and $G''(u, \cdot)$, for the moment whenever $u\in W^{1,p^+}({U },\R^N)$. The statement is formulated in a slightly more general way, in  view of the later application in Lemma \ref{G''min-aux}.
 \begin{lemma}\label{innreg} Let  $u\in L^1({U },\R^N)$,  $V\in \mathcal A(U)$ and assume that there exists a compact set $K\subset V$ such that $u\in W^{1,p^+}(V\setminus K,\R^N)$. Then 
  \begin{equation}\label{regolmodif}
    G'(u,V)=\sup\{G'(u,V'):\, V'\Subset V\}, \qquad   G''(u,V)=\sup\{G''(u,V'):\, V'\Subset V\}.
\end{equation}
In particular,  for every $u\in W^{1,p^+}(U,\R^N)$ we have that $G'(u,\cdot)$ and $G''(u,\cdot)$ are inner regular  increasing set functions on $\mathcal A( {{U }}).$ 
Moreover, $G''(u,\cdot)$ is subadditive and superadditive for every $u\in W^{1,p^+}({U },\R^N)$.
\end{lemma}
\begin{proof} Let $u\in L^{1}({U },\R^N)$ and let  $V,K$ be as in the statement. We show that $G'$ satisfies the  first  property in \eqref{regolmodif} (the proof for $G''$ is analogous).
Let $K'$ be a compact set  such that $K\subseteq K'\subseteq  V$ and choose $V', V''\in \mathcal{A}({V })$  such that $K'\subset V''\Subset V' \Subset V$. We apply Lemma \ref{subadditivity} with  the triple $V''$, $V'$ and $W=V\setminus K'$, and observe that $V''\cup W=V$. Also using \eqref{altop+} on $W$,  we then have
\begin{align*}
    G'(u,V)&\le G'(u,V')+G''(u,V\setminus K')\le \sup\{G'(u,V'):\, V'\Subset V\}+G''(u,V\setminus K')\\
    &\le \sup\{G'( u,V'):\, V'\Subset V\}+\beta\displaystyle \int_{V\setminus K'} \left(1+\left|Du\right|^{p^+}\right)\, \mathrm{d}x.\
\end{align*}
Then, as $|V\setminus K'|\to 0$, we get 
 \begin{equation*}    G'(u,V)\le \sup\{G'(u,V'):\, V'\Subset V\}.
\end{equation*}
Since the reverse inequality is trivial, we obtain that the first equality in \eqref{regolmodif} holds.
If $u\in W^{1,p^+}(U, \mathbb R^N)$, \eqref{regolmodif} follows on every $V\in \mathcal A(U)$ by simply taking $K=\emptyset$. 

Now, let $u\in W^{1,p^+}({U },\R^N)$. Given $V,W\in\A({U })$, from  \eqref{eq: subaddG''} and the inner regularity of $G''(u,\cdot)$ we deduce that \begin{equation*}
    G''(u,V\cup W)=\sup\{G''(u,V'\cup W)\mid V'\Subset V\}\le G''(u,V)+G''(u,W)
\end{equation*}
that is the subadditivity of $G''(u,\cdot)$.

As for the superadditivity, given  $V,W\in\A({U })$ such that  $V\cap W=\varnothing$, consider $V'$, $W' \in \mathcal{D}$ with $V'\Subset V$, $W'\Subset W$. Using  the superadditivity of $G'(u,\cdot)$ of  Lemma \ref{subadditivity}  and \eqref{epsk}, we have that 
\[
G''(u, V\cup W)\geq G'(u, V \cup W)\geq G'(u, V)+G'(u, W)\geq G'(u, V')+G'(u, W')=G''(u, V')+G''(u, W')
\]
and we can conclude by the inner regularity of $G(u, \cdot)$ whenever $u\in W^{1,p^+}({U },\R^N)$.
\end{proof}
We are now in a position to give a partial integral representation result for $G''(u, V)$, namely when $u\in W^{1,p^+}({U },\R^N)$.

\begin{lemma}\label{limsup}
    Under the previous  hypotheses, there exists a Carathéodory function $\gamma\colon {U }\times {\bf M}^{N\times d }\to [0,+\infty)$, even in the second variable, such that  
\begin{enumerate}
    \item for every  $u\in W^{1,p^+}({U },\R^N)$ and for every  $V\in \mathcal{A}({U })$  it holds
    \begin{equation}\label{reprlimsup}
    G''(u,V)=  \displaystyle \int_V \gamma (x, D u)\,\mathrm{d}x;
    \end{equation}
    \item  $\alpha\left(|\Sigma|^{p^-}-1\right)\leq \gamma(x,\Sigma)\leq \beta\left(1+|\Sigma|^{p^+}\right)$ for a.e. $x\in {U }$ and for every $\Sigma\in {\bf M}^{N\times d}$;
    \item $\gamma(x,\cdot)$ is convex for a.e. $x\in {U }$.
  \end{enumerate}
\end{lemma}

\begin{proof}
In order to apply \cite[Theorem 9.1]{BDF}, in the following we verify its hypotheses. \begin{enumerate}
    \item {\it locality}: $G''(\cdot,V)$ is trivially local for each $V\in\mathcal A({U })$;
    \item {\it measure property}: for every $u\in  W^{1,p^+}({U },\R^N)$, by Lemmas \ref{subadditivity} and \ref{innreg}, we have that  $G''(u,\cdot)$ is increasing, inner regular, subadditive and superadditive; thus, by De Giorgi-Letta criterion,  $G''(u,\cdot)$ is the restriction on $\mathcal A({U })$ of a Borel measure.
    \item {\it growth condition from above}: it is ensured by \eqref{altop+}.
    \item {\it translation invariance in $u$}: trivially we have $G''(u+z,V)=G''(u,V)$ for any $u\in  W^{1,p^+}({U },\R^N)$, $z\in\R^N$ and $V\in\mathcal A({U })$.
    \item {\it lower semicontinuity}: for each $V\in\mathcal A({U })$, since $G''(\cdot,V)$ is $L^1({U },\R^N)$-lower semicontinuous, it is also lower semicontinuous on $  W^{1,p^+}({U },\R^N)$.
\end{enumerate}
By applying \cite[Theorem 9.1]{BDF}, we get that there exists a Carath\'eodory function $\gamma\colon {U }\times {\bf M}^{d\times N }\to [0,+\infty)$ satisfying
$\gamma(x,\Sigma)\leq \left(1+|\Sigma|^{p^+}\right)$ for a.e. $x\in {U }$ and for every $\Sigma\in {\bf M}^{N\times d}$
and such that
 \begin{equation*}\label{represGtildep+}
    G''(u,V)= \int_V \gamma(x,D u(x))\,\mathrm{d}x \qquad  \forall u\in W^{1,p^+}({U },\R^{ N}), \qquad  \forall\, V\in \A({U }).
        \end{equation*}
        Such a function $\gamma$ is defined by
\begin{equation}\label{blowupgamma}
 \gamma(x,\Sigma)=\limsup_{r\to 0^+}\frac{G''(u_\Sigma , B_{r}(x))}{|B_{r}(x)|} \qquad \forall x\in {U },\, \forall\,\Sigma \in {\bf M}^{N\times d},
\end{equation} 
where $u_\Sigma$ is  the affine function $y\mapsto \Sigma y$.
In particular, thanks to \eqref{bassop-}, we have that 
\begin{equation}\label{alfa} \alpha\left(| \Sigma|^{p^-}-1\right)
\leq \gamma(x,\Sigma)
\end{equation}
for a.e. $x\in {U }$ and for every $\Sigma\in {\bf M}^{N\times d}$. By definition of $\gamma$, it easily follows that $\gamma(x,\cdot)$ is even for every $x\in U$.
Finally,  thanks to the convexity property of $G_{\varepsilon_{k}}$, we have that  the function $\Sigma\mapsto  G''(u_{\Sigma},B_{r}(x))$ is convex. 
Being   $\gamma(x,\cdot)$  the $\limsup$ of convex functions, we conclude that   $\gamma(x,\cdot)$ is convex for a.e. $x\in {U }$.
\end{proof}

As it usually happens for $\Gamma$-convergence problems, the function $\gamma(x,\Sigma)$ defined in \eqref{blowupgamma} can be equivalently characterized by the asymptotical behavior of minimum problems. Even though it is not necessary for the present section, we will discuss this property (together with a result about the convergence of minima) in the   Proposition \ref{G''min}, which will reveal itself to be crucial in our treatment of the stochastic case (see Section \ref{stochastic}). 
Before we need  the following auxiliary lemma that provides an additivity property on $W^{1,p^-}(U , \R^N)$, provided that a boundary condition is imposed. We recall, indeed, that so far we know  that  $G''(u,\cdot)$ is  additive on $\mathcal A(U)$ only when $u\in W^{1,p^+}(U , \R^N)$. 
\begin{lemma}\label{G''min-aux} Let $\Sigma\in {\bf M}^{N\times d}$. Under the assumptions of this section, the following finite  additivity property holds: if $A\in \mathcal A(U)$,  $k\in \N$,  $\{V_1,\cdots V_k\}$ are open sets and $v_i\in W^{1,p^-}(V_i , \R^N)$  are  such that 
\begin{enumerate}
\item $V_i \subset A$ for every $i\in \{1,\cdots,k\}$ and $V_i\cap V_j=\emptyset$ for every $i\not=j$; 
\item  $|\partial V_i|=0$;
\item  $v_i=u_\Sigma$ in a neighbourhood of $\partial V_i$;
\end{enumerate}  then, defined 
\begin{equation}\label{tildev}\tilde v(x)=\begin{cases} v_i(x) \quad \hbox{ in } V_i ,\\
u_{\Sigma} \quad \hbox{ otherwise in } U , \end{cases} 
\end{equation}
 we have that  $\tilde v\in  W^{1,p^-}(U , \R^N)$ and 
\begin{equation}\label{finiteadditivity}
G''( \tilde v, A)=\sum_{i=1}^k G''(v_i,  V_i )+G''\left(\tilde v, A\setminus \bigcup_{i=1}^k  \overline {V_i}\right)\,.
\end{equation}
\end{lemma}
\begin{proof}
First of all,  we show  \eqref{finiteadditivity} for $k=1$:  let $V\subset A$ be an open set with $|\partial V|=0$ and let $v\in W^{1,p^-}(V , \R^N)$ be such that  $v=u_\Sigma$ in a neighbourhood  $Z$ of $ \partial V$. We claim  that 
\begin{equation}\label{additivitytilde}
G''( \tilde v, A)=G''( v,  V )+G''(\tilde v,A\setminus \overline{V}).
\end{equation}
Indeed, fix $\sigma >0$ and observe that, since $|\partial V|=0$, we can take an open set $Z_\sigma\subseteq \R^d $ and  two open subsets $V'\Subset  V\Subset V''$ such that 
$\partial V \subset Z_\sigma$, $Z_\sigma \cap V'' \subset Z$ ,  $|Z_\sigma \cap V''|\le \sigma$ and  $A= V' \cup (Z_\sigma\cap A)$.
By using the monotonicity of $G''$,  the subadditivity property \eqref{eq: subaddG''} when  $W=Z_\sigma\cap A$ and   the locality property of $G''$, we obtain   that
\[
G''(\tilde v , A)= G''(\tilde v , V' \cup (Z_\sigma \cap A))  \leq  G''( v ,  V )+G''(  \tilde v , Z_\sigma \cap A).
\]
Now, taking into account that  $v=u_\Sigma$ in $Z_\sigma\cap V''$,  the representation result \eqref{reprlimsup} can be applied to $G''(v, Z_\sigma\cap V'')$: splitting the integral,  we obtain  that
   \[
G''(\tilde v, A)\leq G''( v,  V )+G''(\tilde v,A\setminus \overline{V})+ \displaystyle \int_{Z_\sigma \cap V''} \gamma (x, \Sigma)\,\mathrm{d}x.\]
As $|Z_\sigma \cap V''|\le \sigma$ and $\sigma$ is arbitrary, we get the $\le$ inequality in \eqref{additivitytilde}. 
For the converse inequality, let  $D_1,D_2\in \mathcal D$ such that $D_1\Subset A\setminus \overline V$ and 
 $D_2\Subset V$. 
Then,  being $G'$ superadditive,  taking also into account the locality of $G''(\cdot,D_1)$ and  \eqref{epsk}, we obtain that 
\begin{align*}G''(\tilde v,A)&\geq G'(\tilde v,D_1\cup  D_2) \geq G'(\tilde v,D_1)+G'(\tilde v,  D_2)=G''( u_{\Sigma},D_1)+G''(\tilde v,  D_2).
\end{align*}
Thanks to the density of  $\mathcal D$, we can pass to the supremum with respect to $D_1\Subset A\setminus \overline V$ and  $D_2\Subset V$ in the above inequality. Since $u_{\Sigma}\in W^{1,p^+}(A\setminus \overline V,\R^N)$ and $\tilde v=u_{\Sigma}\in W^{1,p^+}(A\setminus K,\R^N)$ with $K=\overline{V\setminus Z}$, we can use Lemma \ref{innreg} to get 
\begin{align*}G''(\tilde v,A)
\geq G''(u_{\Sigma}, A\setminus \overline V)+G''(\tilde v,  V)=G''(\tilde v, A\setminus \overline V)+G''(\tilde v,  V),
\end{align*}
which  concludes the proof of \eqref{additivitytilde}. Finally, by  applying repeatedly   property \eqref{additivitytilde} with  $A'_j=A\setminus\left( \overline {V_1}\cup \cdots  \cup \overline{V_{j}}\right) $ and $V=V_{j+1}$, 
we get \eqref{finiteadditivity}.
\end{proof}

\begin{prop}\label{G''min}
    Let $m_{G''}$ and   $m_{G_{\varepsilon_k}}$  be given by  $\eqref{eq: minimointornobordo}$ and let $u_\Sigma$ denote the affine function $y\mapsto \Sigma y$. Under the assumptions of this section, the following hold:
   \begin{enumerate}
\item 
for every $\Sigma \in \mathbf{M}^{N\times d}$ and for a.e.  $x \in U $ we have
\begin{equation}\label{eq:cellformula1}
\gamma(x,\Sigma)
=
\lim_{r\to 0^+}
\frac{m_{G''}(u_\Sigma, Q_r(x))}{|Q_r(x)|},
\end{equation}
where $\gamma$ is given by  \eqref{blowupgamma};
\item 
for every $\Sigma \in \mathbf{M}^{N\times d}$ and for a.e $x \in U $, we have
\begin{equation}\label{eq:convminima}
\limsup_{r\to 0^+}
\limsup_{k\to +\infty}
\frac{m_{G_{\varepsilon_k}}(u_\Sigma, Q_r(x))}{|Q_r(x)|}\le \gamma(x, \Sigma)\le \liminf_{\theta \to 1^-}\liminf_{r\to 0^+}
\limsup_{k\to +\infty}
\frac{m_{G_{\varepsilon_k}}(u_\Sigma, Q_{\theta r}(x))}{|Q_{\theta r}(x)|}.
\end{equation}
\end{enumerate}

\end{prop}
\begin{proof}
We start with \eqref{eq:cellformula1}. We claim that
\begin{equation}\label{mdelta}
\lim_{\delta\to 0^+}
m_{G''}(u_\Sigma, Q_{r-\delta}(x))
=
m_{G''}(u_\Sigma, Q_r(x)).
\end{equation}
This is the content of \cite[Lemma~6]{bflm}, adapted to our context.
To prove this claim, fix $\Sigma \in \mathbf{M}^{N\times d}$ and,  given $\delta,\sigma>0$, we choose $v\in W^{1,p^-}({U }, \mathbb R^N)$ such that
$v=u_\Sigma$ on a neighbourhood  of $\partial Q_{r-\delta}(x)$ and
\begin{equation}\label{eq:minH}
m_{G''}(u_\Sigma, Q_{r-\delta}(x))
\ge
G''(v, Q_{r-\delta}(x))-\sigma.
\end{equation}
Set
\begin{equation*}
w=
\begin{cases}
v & \text{in } Q_{r-\delta}(x),\\[2pt]
u_\Sigma & \text{in } Q_r(x)\setminus\overline{Q_{r-\delta}(x)}.
\end{cases}
\end{equation*}
For all $\theta>0$ sufficiently small, one has $w=u_\Sigma$ in $Q_r(x)\setminus\overline{Q_{r-\delta-2\theta}(x)}$. Then, by \eqref{eq: subaddG''} with $V'=Q_{r-\delta-\theta}(x)$, $V=Q_{r-\delta}(x)$ and $W=Q_r(x)\setminus\overline{Q_{r-\delta-2\theta}(x)}$  we have
\begin{equation*}
m_{G''}(u_\Sigma, Q_r(x))
\le
G''(w, Q_r(x))
\le
G''(v, Q_{r-\delta}(x))
+
\beta \max\left\{|\Sigma|^{p^-}, |\Sigma|^{p^+}\right\}(r^d-(r-\delta-2\theta)^d)
\end{equation*}
Using \eqref{eq:minH}, letting $\sigma$, $\theta$ and eventually $\delta \to 0^+$, we get
\begin{equation}\label{eq:lscminH}
\liminf_{\delta\to 0^+}
m_{G''}(u_\Sigma, Q_{r-\delta}(x))
\ge
m_{G''}(u_\Sigma, Q_r(x)).
\end{equation}

Conversely, given $\eta>0$, choose $v\in W^{1,p^-}(U , \mathbb R^N)$ such that
$v=u_\Sigma$ in a neighborhood of $\partial Q_r(x)$ and
\begin{equation*}
m_{G''}(u_\Sigma, Q_r(x))
\ge
G''(v, Q_r(x))-\eta.
\end{equation*}
Let $\delta_0>0$ be small enough so that $v=u_\Sigma$ on $Q_r(x)\setminus Q_{r-\delta_0}(x)$.
Then, for all $0<\delta<\delta_0$ one has  $v=u_\Sigma$ in a neighborhood of $\partial  Q_{r-\delta}(x)$, hence
\begin{equation*}
m_{G''}(u_\Sigma, Q_{r-\delta}(x))
\le
G''(v, Q_{r-\delta}(x))
\le
G''(v, Q_r(x))
\le
m_{G''}(u_\Sigma, Q_r(x))+\eta.
\end{equation*}
Letting $\delta\to 0^+$ and then $\eta\to 0^+$, we get
\begin{equation*}
\limsup_{\delta\to 0^+}m_{G''}(u_\Sigma, Q_{r-\delta}(x))
\le
m_{G''}(u_\Sigma, Q_r(x)),
\end{equation*}
which, combined with \eqref{eq:lscminH}, proves the claim.

Once \eqref{mdelta} is established, the remainder of the proof of  \eqref{eq:cellformula1} can now be carried out as in \cite[Lemma 3.3 and 3.5]{bfm}, provided that the following remarks are taken into account.  

 First of all,   \cite[Lemma 3.3]{bfm} holds under our standing assumptions since  $G''(\cdot,V)$ is $L^1({U },\R^N)$-lower semicontinuous, satisfies the coercivity condition \eqref{bassop-} and the finite additivity property  \eqref{finiteadditivity}. 
Moreover, we can show  the result of \cite[Lemma 3.5]{bfm} with $u=u_{\Sigma}$ since 
\begin{itemize}
\item $G''(u_{\Sigma},\cdot)$ satisfies the required assumptions of being the restriction to $\mathcal A(U)$ of a Radon measure, which is absolutely continuous  with respect to the Lebesgue measure (thanks to the growth condition from the above  in  \eqref{altop+});
\item  the continuity property   \eqref{mdelta} holds. 
\end{itemize}

\vspace{3mm}

In order to show the last point \eqref{eq:convminima},  we start claiming that for every $Q_r(x)\subseteq U$ and for every $v\in L^1(U,{\bf M}^{d\times N})$, if  $v=u_\Sigma$ in a neighborhood of $\partial Q_r(x)$, then  
\begin{equation}\label{eq:uguali}
G'(v, Q_r(x))=G''(v, Q_r(x)).
\end{equation}
To prove this, notice that,   by  \eqref{eq: subadd}, we have
\[
G'(v, Q_{r+\theta}(x))\le G'(v, Q_r(x))+G''(u_\Sigma, Q_{r+\theta}(x)\setminus \overline{Q_{r-\theta}(x)})
\] 
so that, by \eqref{altop+}, $r$ is a right continuity point for the increasing set function $s\mapsto G'(v, Q_s(x))$. Assume, by contradiction,   that  \eqref{eq:uguali} does not hold.  Then
$G'(v, Q_r(x))<+\infty$, which implies  that  $v\in W^{1,p^-}(Q_r(x), \R^N)$, and there exists $0<\sigma< G''(v, Q_r(x)) -G'(v, Q_r(x))$.
By applying the  continuity property \eqref{mdelta}, we can  find $\theta>0$ such that
\begin{equation}\label{finitezza}
G'(v, Q_{r+\theta}(x))\leq G'(v, Q_r(x))+\sigma<G''(v, Q_r(x))
\end{equation}
and $v=u_{\Sigma}$ on $Q_{r+\theta}(x)\setminus Q_{r}(x)$.
For $\tilde v$ as in \eqref{tildev} with $V_1=Q_{r+\theta}(x)$, 
by combining   the subadditivity \eqref{eq: subadd} with  \eqref{finitezza} and \eqref{additivitytilde}, we obtain  that 
\begin{align*}
\nonumber G'( \tilde v, {U })&=G'( \tilde v,  Q_{r+\frac \theta 2 }(x) \cup (U\setminus \overline{Q_r(x)}))\leq G'( \tilde  v,   Q_{r+\theta}(x))+G''(u_{\Sigma}, U \setminus  \overline {Q_{r}(x)})\\
&=  G'(  v,   Q_{r+\theta}(x))+G''(u_{\Sigma}, U \setminus  \overline {Q_{r}(x)})\\
\label{interme}  &< G''(v, Q_r(x))+G''(u_{\Sigma}, U \setminus  \overline {Q_{r}(x)})=G''(\tilde v, U)
\end{align*}
which implies a contradiction since, thanks to \eqref{epsk}, we have that  $G''(\tilde v, U)=G'(\tilde v, U)$.
Hence  \eqref{eq:uguali} holds. 

\noindent Now, given $\Sigma$ and $\theta\in (0,1)$,  we fix $x$ such that \eqref{eq:cellformula1} holds. For every $k\in \N$ let  $v_k\in L^1(U,{\bf M}^{d\times N})$ be such that    $G_{\varepsilon_k}(v_k,Q_{\theta r}(x))\leq m_{G_{\varepsilon_k}}(u_{\Sigma},Q_{\theta r}(x))+2^{-k}$. The upper bound in \eqref{growth-p-p+} implies that 
\[ m_{G_{\varepsilon_k}}(u_{\Sigma},Q_{\theta r}(x)) \leq G_{\varepsilon_k}(u_{\Sigma},Q_{\theta r}(x))\leq \beta \max \{|\Sigma|^{p^+}, |\Sigma|^{p^-}\}.\]
Hence, thanks to the lower bound in \eqref{growth-g},  we get that 
 $$\sup_{k\in \N}  \int_{Q_{\theta r}(x)}  \alpha \left(|D v_{{k}}|^{p^-}-1\right)\,\mathrm{d}x\leq \sup_{k\in \N} G_{\varepsilon_k}(v_k,Q_{\theta r}(x))\leq  \beta \max\{|\Sigma|^{p^+}, |\Sigma|^{p^-}\},$$
 i.e. the sequence $\{Du_k\}_{k}  $ is bounded in $L^{p^-}(Q_{\theta r}(x) , {\bf M}^{N\times d}).$
This implies that   $\{v_k\}$ is weakly compact in $W^{1,p^-}(Q_{\theta r}(x),\R^N)$. If we extend $v_k$ with $u_\Sigma$ in the rest of $U$, we have that, up to a subsequence, $v_{k_j}\to v$ in $L^1(U,\R^N)$, $v\in W^{1,p^-}(U ,\R^N)$  and $v=u_\Sigma$  on  $ {U }\setminus \overline{Q_{\theta r}(x)}$. We set $\widetilde v_k=v_{k_j}$ if $k_j \le k <k_{j+1}$ and we have, also using \eqref{eq:uguali},
\[
\begin{split}
G''(v, Q_r(x))&=G'(v, Q_r(x))\le \liminf_{k \to +\infty} G_{\varepsilon_k}(\widetilde v_k,Q_{r}(x))\\
&\le
\liminf_{j \to +\infty} G_{\varepsilon_{k_j}}(v_{k_j},Q_{r}(x))
\le 
\limsup_{k\to+\infty} G_{\varepsilon_k}(v_k,Q_{r}(x)).
\end{split}
\] 
Since $v=u_\Sigma$ in $Q_r(x)\setminus \overline{Q_{\theta r}(x)}$,  we then have
\[
  m_{G''}(u_\Sigma ,Q_r(x))\leq   G''(v,Q_r(x))\le\limsup_{k\to\infty} G_{\varepsilon_k}(v_k,Q_{\theta r}(x))+ \beta \max\left\{|\Sigma|^{p-}, |\Sigma|^{p^+}\right\}(1-\theta^d) r^d.
\]

On dividing by $|Q_r(x)|=\frac 1{\theta^d}|Q_{\theta r}(x)|$, taking the liminf as $r \to 0^+$ and then as $\theta \to 1^-$, since $x$ satisfies  \eqref{eq:cellformula1}, we get
\[
 \gamma(x, \Sigma)\le \liminf_{\theta \to 1^-}\liminf_{r\to 0^+}
\limsup_{k \to +\infty}
\frac{m_{G_{\varepsilon_k}}(u_\Sigma, Q_{\theta r}(x))}{|Q_{\theta r}(x)|}.
\]

\vspace{3mm}

Conversely, given $\sigma>0$, choose $v$ such that
$v=u_\Sigma$ in a neighborhood  of $\partial Q_r(x)$ and
\begin{equation*}
m_{G''}(u_\Sigma, Q_r(x))
\ge
G''(v, Q_r(x))-\sigma.
\end{equation*}
Let $\delta_0>0$ be small enough so that $v=u_\Sigma$ on $Q_r(x)\setminus Q_{r-\delta_0}(x)$. Now, if $\{v_k\}$ is  a recovery sequence for $G''(v, Q_r(x))$, and $M>\sqrt d|\Sigma|$, the truncated functions $\{v_k^M\}$ converge to $v^M=u_\Sigma$ in $L^{p^+}(Q_r(x)\setminus Q_{r-\delta_0}(x))$.  
Furthermore, as already noticed, $G_{\varepsilon_k}(v_k^M,Q_r(x))\le G_{\varepsilon_k}(v_k,Q_r(x))$. 
For each $0<\delta<\delta_0$ we can then exploit the fundamental estimate in  Lemma \ref{fundest} with $V'=Q_{r-\delta}(x)$, $V=Q_r(x)$ and $W=Q_r(x)\setminus \overline{Q_{r-\delta/2}(x)}$. We find $\varphi_{k}\in C^{\infty}_0(Q_r(x))$ such that, for $z_k=\varphi_{k}v_k^M+(1-\varphi_{k})u_\Sigma$, it holds
\[
 \limsup_{k\to +\infty}G_{\varepsilon_k}(z_k,Q_r(x))\le (1+\sigma)(G''(v, Q_r(x))+ \limsup_{k\to\infty} G_{\varepsilon_k}(u_\Sigma, Q_r(x)\setminus Q_{r-\delta}(x)))+\sigma.
\]
 Passing to the minima in the left-hand side  and  exploiting the growth condition \eqref{growth-p-p+}, we obtain 
\[
    \limsup_{k\to +\infty}m_{G_{\varepsilon_k}}(u_\Sigma,Q_r(x))\le (1+\sigma)\left(m_{G''}(u_\Sigma,Q_r(x))+\sigma \right)+\beta \max\left\{|\Sigma|^{p-}, |\Sigma|^{p^+}\right\}(r^d-(r-\delta)^d)+\sigma.
\]
As $\delta$ and $\sigma\to 0^+$, we get that 
\begin{equation*}\label{limsupm}
\limsup_{k\to +\infty}m_{G_{\varepsilon}}(u_\Sigma,Q_r(x))\le m_{G''}(u_\Sigma,Q_r(x)).
\end{equation*} 
On dividing by $|Q_r(x)|$ and using again  \eqref{eq:cellformula1}, we get the first inequality in \eqref{eq:convminima}, which concludes the proof.
\end{proof}

\begin{rem}\label{no-periodicity3}  In the rest of this section, starting from next proposition, all the results will rely
on the periodicity assumption on $g(\cdot,\Sigma)$. \end{rem}

\begin{prop}\label{gammaN}
The function $\gamma$ defined by \eqref{blowupgamma} can be chosen to be independent of the first variable  and 
 satisfies the reinforced-$\Delta_2$ property  \eqref{Delta2gamma}.
Moreover,     it holds
\begin{equation}\label{youngfunc} C \min \{|\Sigma|^{p^+}, |\Sigma|^{p^-}\}\leq  \gamma( \Sigma)\leq \max\{|\Sigma|^{p^+}, |\Sigma|^{p^-}\} \qquad \hbox{ for every }  \Sigma\in {\bf M}^{N\times d}\end{equation}
where $C$ is a positive constant.
In particular,  $\gamma$ is an $N$-function. 
\end{prop}
\begin{proof}
For $\Sigma \in {\bf M}^{N\times d}$ set as usual $u_\Sigma(x)=\Sigma x$. Exploiting the inner regularity of $G''(u_\Sigma, \cdot)$ and the $1$-periodicity of $g$, the proof that $\gamma$ is independent of $x$ can be carried out exactly as in \cite[Proposition 21.11]{BDF}.
 Now we will show that $\gamma$ satisfies the remaining properties.
 Indeed let $A\in {\bf M}^{N\times N}$,  and let $\{u_k\}\subseteq W^{1,p^+}({U },\R^N)$ be such that $u_k\to u_\Sigma $      in $L^1({U },\R^N)$ as $k\to +\infty$ and         
$$
\gamma(\Sigma)|{U }|=G''(u_\Sigma,{U })=\limsup_{k \to +\infty} \int_{U }g\left(\frac{x}{\varepsilon_k}, Du_k \right) \, \mathrm{d}x.$$
 Since $g$ is radially increasing in the last variable and $|A\Sigma|\le \|A\|\,|\Sigma|$, thanks to  the Cauchy-Schwarz inequality,     we can apply  \eqref{cons2} with $s=\|A\|$ and $t=|Du_k(x)|$ and  get that, for a.e. $x\in U$ and for every $k\in \mathbb N$, it holds
  \[g\left(\frac{x}{\varepsilon_k}, ADu_k\right) \leq \max\{ \|A\|^{p^+},\|A\|^{p^-}  \} g\left(\frac{x}{\varepsilon_k}, Du_k\right).\]
 Hence we  obtain that
     \begin{equation}\label{stima}
    \begin{split}
    \gamma(A \Sigma)|{U }|&=\int_{{U }} \gamma(A \Sigma)\mathrm{d}x= G''(Au_\Sigma,{U })\\
     &\leq  \limsup_{k \to +\infty} \int_{U }g\left(\frac{x}{\varepsilon_k}, ADu_k\right) \, \mathrm{d}x \\
     &\leq  \max\{ \|A\|^{p^+},\|A\|^{p^-}  \} \limsup_{k}   \int_U  g\left(\frac{x}{\varepsilon_k}, Du_k\right)\, \mathrm{d}x\\
     &=\max\{ \|A\|^{p^+},\|A\|^{p^-}  \}\gamma(\Sigma)|{U }|,
     \end{split}
     \end{equation}
    i.e. $\gamma$ satisfies the reinforced-$\Delta_2$ property. 
     In order to show that $\gamma$ satisfies  the right hand estimate in \eqref{youngfunc},   with a similar argument and using that $\|I_N\|=1$ where $I_N$ is the identity matrix, we have
     $$
\gamma( \Sigma)|{U }|=G''(u_{\Sigma},{U })\leq  \limsup_{k\to +\infty} \int_{{U }} g\left(\frac{x}{\varepsilon_k}, \Sigma\right) \, \mathrm{d}x\leq \max\{|\Sigma|^{p^+}, |\Sigma|^{p^-}\} \limsup_{k\to +\infty} \int_{{U }} g\left(\frac{x}{\varepsilon_k}, 1\right) \, \mathrm{d}x$$     
     which, using (A0) and dividing by $|U|$, entails
 $$
\gamma( \Sigma)\leq \beta\max\{|\Sigma|^{p^+}, |\Sigma|^{p^-}\}\,.
$$
 This also gives  $\gamma(0)=0$. Hence, the left hand estimate in \eqref{youngfunc} has to be proved only for   $ \Sigma\in {\bf M}^{N\times d}\setminus\{0\}$. Firstly, we notice that, by applying \eqref{stima} to  $t\Sigma$ and $A=\frac 1 t I_N$, with $t>0$, we obtain  
$$
    \gamma( \Sigma)\leq \max\{ t ^{-p^+},  t ^{-p^-}  \}\gamma(t \Sigma),
    $$
that is 
\begin{equation}\label{inbasso}
\gamma(t\Sigma )\geq \min\{ t^{p^+}, t^{p^-}\}\gamma(\Sigma) \qquad \forall{ t>0}, \forall\, \Sigma \in {\bf M}^{N\times d}.
\end{equation}  
Now, \eqref{alfa} gives $\gamma(\Sigma) \ge \alpha$ whenever $|\Sigma|\ge 2$. Hence, for arbitrary $\Sigma \in {\bf M}^{N\times d}\setminus\{0\}$, taking $t=\tfrac{|\Sigma|}2$ in \eqref{inbasso} we have
\[
\gamma(\Sigma)=\gamma\left(\frac{|\Sigma|}2\frac{2\Sigma}{|\Sigma|}\right)\ge\frac{\alpha}{2^{p^+}}\min\{|\Sigma|^{p^+}, |\Sigma|^{p^-}\}\,,
\]
as required.
Thanks also to Lemma \ref{limsup}, we can conclude that $\gamma$ is an $N$-function.
\end{proof}
 
Now,  from Corollary \ref{senzaLinfty}, we may  deduce the $\Gamma$-limsup inequality on $W_{\gamma}({U },\R^N)$.
\begin{cor}\label{necessario}Assume \eqref{epsk} and let  $\gamma$ be  defined by \eqref{blowupgamma}.  
Then 
  \begin{equation}\label{reprparziale}
    G''(u,V)\leq  \displaystyle \int_V \gamma ( D u)\,\mathrm{d}x, \qquad \forall\, u\in W_{\gamma}({U },\R^N), \forall\, V\in \mathcal{A}({U }).
    \end{equation}
  
    \end{cor}
    \begin{proof}
Let $u\in W_{\gamma}({U },\R^N)$. Thanks to Proposition \ref{gammaN}, we have that   $\gamma$ satisfies the reinforced-$\Delta_2$ property. Hence,  by applying Corollary \ref{senzaLinfty}, there exists a sequence $\{u_k\}_k\subseteq  C^{\infty}(\overline {U },\R^N)$ such that $u_k\to u$ in  $W^{1,1}({U },\R^N)$ and  
\begin{equation}\label{eq: L1}
\int_{{U }} \gamma (Du_k) \mathrm{d}x \to  \int_{{U }} \gamma (Du) \mathrm{d}x,
\end{equation}
as $k\to +\infty$.
We claim that $$\int_{V} \gamma (Du_k) \mathrm{d}x \to  \int_{V} \gamma (Du) \mathrm{d}x$$ for any $V\in \mathcal A({U })$.
Indeed, since $\gamma$ is continuous, we have that  $\gamma(Du_k)\to \gamma(Du)$ in measure. Since  \eqref{eq: L1} reads as $\|\gamma(Du_k)\|_{L^1({U })}\to \|\gamma(Du)\|_{L^1({U })}$, by a standard result in measure theory (see, for instance, \cite[Proposition 1.33]{AFP}),  it follows that  $\|\gamma(Du_k)-\gamma(Du)\|_{L^1({U })}\to 0$. This implies the claim. Finally, since $G''$  is lower semicontinuous w.r.t.\ the $L^1$-convergence and admits the representation formula \eqref{reprlimsup} in $W^{1,p^+}({U },\R^N)$, we get
 \[G''(u,V) \leq \liminf_{k\to +\infty}  G''( u_k,V)=\liminf_{k\to +\infty} \int_{V} \gamma (Du_k) \mathrm{d}x= \int_{V} \gamma (Du) \mathrm{d}x. \] \end{proof}
 
In the next lemma, periodicity of $g(\cdot, s)$ has a key role, as it entails translation invariance of $G'$, which is crucial for the proof below.
  \begin{lemma}\label{lowerbound}
Assume \eqref{epsk} and let  $\gamma$ be  defined by \eqref{blowupgamma}. Then, it holds
  \begin{equation}\label{lowbound}G'(u,V)\geq \int_V \gamma(Du(x))\,\mathrm{d}x\qquad  \forall\, u\in W^{1,1}({U },\R^{ N}), \forall\, V\in \mathcal \A({U }).
  \end{equation}

\end{lemma}
  \begin{proof} 
 Without loss of generality, we can assume that $G'(u,V)<+\infty$. Following the lines of \cite[Theorem 24.1]{DM} and thanks to   a translation argument,  it  follows   that,  for every $y\in \mathbb R^d$ and for every $V,D\in \mathcal A({U })$,  if $D\Subset V+y$ then $$ G'(\tau_y u, D)\leq  G'( u, V) $$ where 
  $\tau_y u(x)=u(x-y)$.
Since $g(x,\cdot) $ is a convex function for a.e. $x\in \R^d$, we have that $G'$ is a convex functional.   Then, proceeding as in Step 1 of \cite[Theorem 14.8]{BDF}, let  $V'\Subset V$ and let $D\in \mathcal D$ such that  $V'\Subset D \Subset V$ 
 and let  $\{\rho_j\}_j$   be a sequence of mollifiers with compact support contained in $B\left(0,\frac 1 j\right)$.  Using Jensen’s inequality,  for $j$ large enough as to have $D\Subset V+y$  for all
 $y\in B(0,\frac 1 j)$, we get that
\begin{equation*}\begin{split}  G'(\rho_j*u,D)&=  G'\left(\int_{B(0,\frac 1 j)} \rho_j(y) u(\cdot-y) \,\mathrm{d}y, D\right)\leq \int_{B\left(0,\frac 1 j\right)} \rho_j(y)  G'(u(\cdot-y),  D)\,\mathrm{d}y\\
&=\int_{B(0,\frac 1 j)} \rho_j(y)  G'(\tau_y u,  D)\,\mathrm{d}y \leq \int_{B\left(0,\frac 1 j\right)} \rho_j(y)  G'(u,  V) \,\mathrm{d}y= G'(u,  V).
 \end{split}
 \end{equation*}
Since $\rho_j*u\in W^{1,p^+}(D, \mathbb R^N)$,  by \eqref{epsk}, we have that 
$$ \int_{V'} \gamma(D\rho_j*u) \mathrm{d}x\leq \int_{D} \gamma(D\rho_j*u) \mathrm{d}x =G''(\rho_j*u,D)=G'(\rho_j*u,D)  \leq G'(u,  V).$$
By using the lower semicontinuity,  when $j\to +\infty$ we get
$$ \int_{V'} \gamma(Du) \mathrm{d}x\leq  G'(u,V) \qquad \forall V'\Subset V.$$
By the arbitrariness of $V'\Subset V$,  we obtain  \eqref{lowbound}.
  \end{proof}

We are now ready to prove the $\Gamma$-convergence of $G_{\varepsilon}$ along a  subsequence satisfying \eqref{epsk}. The need for passing to a subsequence will be removed owing to the Lemma \ref{cell}.

\begin{lemma}\label{estratte}
      Let $\varepsilon_k$ be as in \eqref{epsk} and let  $\gamma$ be defined by \eqref{blowupgamma}.  Then 
  \begin{equation}\label{gammalimlemma}
     G'(u,V) = G''(u, V) = \int_{V} \gamma (Du)\,\mathrm{d}x
  \end{equation}     
for every $V\in\A({U })$ and every $u\in  W_{\gamma}({U },\R^N)$.
Furthermore, if $u\in L^1({U },\R^N)\setminus W_\gamma ({U },\R^N)$, then \begin{equation}\label{gammainf}
    \Gamma\left(L^1\right)\hbox{-}\lim_{k\to\infty}G_{\varepsilon_{k}}(u,{U }) = +\infty.
\end{equation} 
In particular, the sequence $\{G_{\varepsilon_{k}}\}$ $\Gamma\left(L^1\right)\hbox{-}$ converges to the functional
 \begin{align*}
G(u):=
\left\{
\begin{array}{l}
\displaystyle\int_{{U }} \gamma(D u)\,\mathrm{d}x\quad \mbox{if }u\in W_{\gamma}({U },\R^N), \\
+\infty \quad \mbox{elsewhere in }L^1({U },\R^N),
\end{array}.
\right.   
\end{align*}
\end{lemma}
\begin{proof}
Assume $u\in  W_{\gamma}({U },\R^N)$. By \eqref{reprparziale} and \eqref{lowbound}, along the subsequence $\varepsilon_{k}$ given by  \eqref{epsk},  we have
    \begin{equation*}
        \int_V\gamma(Du)\,\mathrm{d}x\le G'(u,V)\le G''(u,V)\le \int_V\gamma(Du)\,\mathrm{d}x,
    \end{equation*}
    for every $V\in\A({U })$. This proves \eqref{gammalimlemma}.
Assume  $u\in L^1({U },\R^N)\setminus W_\gamma ({U },\R^N)$. To prove \eqref{gammainf}, it suffices to show that $G'(u,{U })=+\infty$. If $u\in L^1({U },\R^N)\setminus  W^{1,p^-}({U },\R^N)$, 
\eqref{bassop-}  immediately gives that $ G'(u,{U })=+\infty$.  Finally, if  $u \in W^{1,p^-}({U },\R^N)\setminus W_\gamma({U }, \R^N)$, by Lemma \ref{lowerbound} we get
\begin{equation*}
    +\infty=\int_{U } \gamma(Du)\,\mathrm{d}x\le G'(u,{U })
\end{equation*}
and we conclude.
\end{proof}

We are now in a position to prove the cell formula \eqref{fhom}.

\begin{lemma}[Cell formula]\label{cell} The function $\gamma$ given by \eqref{blowupgamma}   satisfies the cell problem formula
\begin{align}\label{reprgamma}
\gamma(\Sigma)=\min_{u\in W_{{\rm per}}^{1,p^-}({(0,1)^d})} \int_{{(0,1)^d}}  g(x, \Sigma+  D u(x))  \, \mathrm{d}x .   
\end{align}
\end{lemma}

\begin{proof}
This is a standard argument in convex periodic homogenization; we give however details as some adaptations are needed for the situation at hand.
We set 
\[
h_J(\Sigma)=\frac1{J^d}\min_{u\in W_{\rm{per}}^{1,p^-}({(0,J)^d})} \int_{{(0,J)^d}} g(x, \Sigma+  D u(x))   \, \mathrm{d}x .
\]
for each $J\in \mathbb{N}$. As the function $(x, \Sigma)\mapsto g(x, \Sigma) $ is convex in $\Sigma$ and $1$-periodic in $x$, we have that, actually, $h_1(\Sigma)=h_J(\Sigma)$ for all $J\in \N$ (see \cite[page 115]{BDF}).
Now, if $u \in W_{{\rm per}}^{1,p^-}({(0,1)^d}, \mathbb R^N)$, we can extend it by periodicity to the whole $\mathbb{R}^d$. Set $v_{k}=\varepsilon_k u\left(\frac x{\varepsilon_k}\right)$, we have that  $v_{k}\to 0$ in $L^1({(0,1)^d}, \mathbb R^N)$ and, by applying Lemma  \ref{estratte},   possibly along a subsequence, we get that, for every $\Sigma\in {\bf M}^{N\times d}$,
\begin{align*}
\gamma(\Sigma)&=\int_{(0,1)^d} \gamma(\Sigma)\,\mathrm{d}x=G''(u_{\Sigma},(0,1)^d)=G'(u_{\Sigma},(0,1)^d)\le \liminf_{k\to +\infty} \int_{(0,1)^d} g\left(\frac x{\varepsilon_k}, \Sigma+ Dv_k\right)\,\mathrm{d}x\\
&=\liminf_{k\to +\infty} \int_{{(0,1)^d}}  g\left(\frac x{\varepsilon_k}, \Sigma+  D u\left(\frac x{\varepsilon_k}\right)\right)\,\mathrm{d}x=\int_{{(0,1)^d}}  g(x, \Sigma+  D u(x))   \, \mathrm{d}x \,,
\end{align*}
where the last limit follows by applying the Riemann-Lebesgue Lemma to the periodic function $x\to  g(x, \Sigma+  D u(x))$. This proves $\gamma(\Sigma)\le h_1(\Sigma)$, as $u\in W_{{\rm per}}^{1,p^-}({(0,1)^d}, \mathbb R^N)$ is arbitrary.

For the converse inequality, let $\Sigma\in {\bf M}^{N\times d}$ and let $\{u_{k}\}$ be a recovery sequence for $u_\Sigma$ on the unit square. Fix $M>\sqrt{d}\,|\Sigma|$ and observe that the vector-valued smooth truncations $u^M_{k}$, given by Lemma \ref{tronc}, satisfy 
\[ \|u_{k}^{M} - u_\Sigma\|_{L^{1}((0,1)^d,\mathbb R^N)} = \|u_{k}^{M} - u_\Sigma^{M}\|_{L^{1}((0,1)^d,\mathbb R^N)}\leq  \|u_{k} - u_\Sigma\|_{L^{1}((0,1)^d,\mathbb R^N)}\]
that implies $\|u_{k}^{M} - u_\Sigma\|_{L^{1}((0,1)^d,\mathbb R^N)} \to 0$ as $k\to +\infty$. With this, since $|Du^M_k| \le |Du_k|$ a.e. and  $g(x, \cdot)$ is radially increasing, we have
\[\begin{split}\gamma(\Sigma)&\leq \liminf_{k\to +\infty} \int_{(0,1)^d}g\left(\frac x{\varepsilon_k},  Du^M_{k}\right)\,\mathrm{d}x\\
& \leq 
\limsup_{k\to +\infty} \int_{(0,1)^d}g\left(\frac x{\varepsilon_k},  Du^M_{k}\right)\,\mathrm{d}x \le \limsup_{k\to +\infty} \int_{(0,1)^d}g\left(\frac x{\varepsilon_k},  Du_{k}\right)\,\mathrm{d}x  =\gamma(\Sigma)\,.
\end{split}
\]
Hence, $\{u_{k}^{M}\}_k$  is itself  a recovery sequence. Without loss of generality, we can then assume that $\|u_{k}\|_{L^\infty((0,1)^d)}\le M$, so that the functions $u_{k}-u_\Sigma$ converge to $0$ in $L^{p^+}({(0,1)^d})$.  With this, for given $\sigma>0$ and $0<\delta<1$, taking the convex combination \[z_{k}(x):=\varphi_k(x)u_{k}(x)+(1-\varphi_k(x))u_\Sigma(x)\,,\]  where $\varphi_k$ are suitable cut-off functions between $V'={(0,1-\delta)^d}$ and $V={(0,1)^d}$, by using the $L^{p^+}$ fundamental estimate of Lemma~\ref{fundest} and \eqref{growth-p-p+}, we get
\[
\limsup_{k\to\infty}\int_{(0,1)^d} g\left(\frac{x}{\varepsilon_k}, Dz_k\right)\,\mathrm{d}x
\le (1+\sigma)\left(\gamma(\Sigma)
+\left(1-(1-\delta)^d\right)\max\{|\Sigma|^{p^-},|\Sigma|^{p^+}\}\right)
+\sigma .
\]
Equivalently,
\[
\limsup_{k\to\infty}\varepsilon_k^d\int_{{\left(0,{\varepsilon_k}^{-1}\right)^d}}g(y, \Sigma +D v_{k}(y))\,\mathrm{d}y
\le  (1+\sigma)\left(\gamma(\Sigma)
+\left(1-(1-\delta)^d\right)\max\{|\Sigma|^{p^-},|\Sigma|^{p^+}\}\right)
+\sigma,
\]
where we set
\[
v_{k}(y)
:= \varepsilon_k z_{k}\!\left(\frac y{\varepsilon_k}\right) - u_\Sigma(y)
\in W^{1, p^-}_0\!\left({\left(0,{\varepsilon_k}^{-1}\right)^d}\right).
\]
Let $J_k$ be the smallest integer greater or equal to $\frac{1}{\varepsilon_k}$,
and extend $v_{k}$ by   $0$ in order to get  a function belonging to  $
W^{1,p^-}_0({(0,J_k)^d}) \subset W^{1,p^-}_{\mathrm{per}}({(0,J_k)^d}).$  Using \eqref{growth-p-p+} to estimate the integral of $g(\cdot,\Sigma)$ on   $ (0,J_k)^d\setminus (0,\varepsilon_k^{-1})^d$,  we get
\[
h_{J_k}(\Sigma)\le (1+\sigma)\gamma(\Sigma)+\sigma
+\left((1+\sigma)\left(1-(1-\delta)^d\right)
+ \frac{J_k^d-\varepsilon_k^{-d}}{J_k^d}\right)
\max\{|\Sigma|^{p^-},|\Sigma|^{p^+}\},
\]
so that
\[
h_1(\Sigma)=\limsup_{k\to +\infty} h_{J_k}(\Sigma)
\le  (1+\sigma)\gamma(\Sigma)
+(1+\sigma)\left(1-(1-\delta)^d\right)\max\{|\Sigma|^{p^-},|\Sigma|^{p^+}\}
+\sigma,
\]
which concludes the proof by arbitrariness of $\sigma$ and $\delta$.
\end{proof}

\begin{proof}[Proof of Theorem \ref{gammamodel}]
 We note that, for any sequence $\{\varepsilon_k\}$ converging to $0$, we can estract a subsequence $\varepsilon_{k_j}$ such that  \eqref{epsk} holds and we can apply Lemma \ref{estratte}.
As the cell formula \eqref{reprgamma}  ensures that  $\gamma$  does not depend on the chosen subsequence,   the Urysohn property of $\Gamma$-convergence entails that the whole  family $\{G_\varepsilon\}$
$\Gamma$-converges to the function $G$ given by \eqref{model-limit}. Furthermore, repeating this argument on each Lipschitz subset of $U$, we deduce, in particular, that \eqref{epsk}  is satisfied by any subsequence $\varepsilon_k$ without further extraction. Hence \eqref{eq: local} follows again  by Lemma \ref{estratte}, formula  \eqref{gammalimlemma}.  This concludes the proof.
\end{proof}

\section{Non-convex  homogenization with generalized Orlicz growth}\label{nonconvexcase}
We now turn to the general case of possibly non-convex periodic bulk integrands $f=f(x,\Sigma)$, with  generalized Orlicz growth $g=g(x,\Sigma)$  as in Section \ref{convex homog}. The key idea is that, in this case, the function $\gamma$ given by \eqref{fhom} provides the growth condition for the $\Gamma$-limit, so that the space $W_\gamma({U },\R^N)$ is the effective domain of the limit energy, represented as an integral functional in this space. 
The starting point is to derive a representation result in the Banach space  $W_\psi({U },\R^N)$, generalizing \cite[Theorem 9.1]{BDF}.

\subsection{A representation result}
In this subsection, we state an integral representation for a functional $H:W_{\psi}({U }, \R^N)\times \A({U }) \to [0,+\infty)$ under some suitable assumptions on $H$, whenever $\psi\colon {\bf M}^{N\times d}\to [0,+\infty)$ is an $N$-function satisfying the reinforced-$\Delta_2$ property.
More precisely, we provide a statement analogous to \cite[Theorem 9.1]{BDF} in the setting of Musielak-Orlicz spaces. Note that the following result  is similar to   \cite[Proposition 4.9]{MM}. The regularity assumption on the set ${U }$ allows us to apply Corollary \ref{PAdensity} in place of the density property required by Mingione and Mucci. We give the proof since we are interested also to the case when the  assumption (H2)' replaces the stronger hypothesis (H2).

 In the following we require that 
 $H$ satisfies the following hypotheses:

\begin{enumerate}
 \item[(H1)]({\it locality})  \( H(\cdot, A) \) is local for every \( A \in \mathcal{A}({U }) \), that is, if \( u, v \in W_{\psi}({U } ,\R^N) \) satisfy \( u = v \ \mathcal{L}^d\text{-a.e. in } A \), then \( H(u, A) = H(v, A) \);
  \item[(H2)] ({\it measure property}) for any \( u \in W_{\psi}({U },\R^N) \),  \( H(u, \cdot) \) is a the restriction to $\mathcal{A}({U })$ of a Borel measure;
  \item[(H2)'] ({\it measure property}) for any \( u \in W^{1,\infty}({U },\R^N) \), \( H(u, \cdot) \) is a the restriction to $\mathcal{A}({U })$ of a Borel measure;
    \item[(H3)] ({\it growth condition}) for every \( u \in W_{\psi}({U },\R^N)  \) and every \( B \in \mathcal{A}({U }) \), it holds
  \[
  0  \leq H(u, B)
    \leq \beta \left( \int_B (1 + \psi( D u))\,\mathrm{d}x  \right);
  \]
\item[(H4)] ({\it translation invariance}) for every \( u \in W_{\psi}({U },\R^N)  \), every $c\in \R^d$ and every \( B \in \mathcal{A}({U }) \), it holds $H(u+c,B)=H(u,B)$.
  \item[(H5)] ({\it lower semicontinuity}) for any \( A \in \mathcal{A}({U }) \),  \( H(\cdot, A) \) is sequentially  lower semicontinuous with respect to the weak convergence of $W_{\psi}({U }, \R^N)$.

\end{enumerate}

\begin{thm}\label{repr}
    Let ${U }\subseteq \R^d$ be a bounded Lipschitz open set and let $\psi\colon {\bf M}^{N\times d}\to [0,+\infty)$ be an $N$-function satisfying the reinforced-$\Delta_2$ property.  Assume that $H\colon W_{\psi}({U },\R^{N})\times \A({U })\to [0,+\infty)$ satisfies \textnormal{(H1), (H2)', (H3), (H4), (H5)}. Then, there exists a Carathéodory function $h: {U }\times {\bf M}^{d\times N }\to [0,+\infty) $ such that 
\begin{enumerate}
    \item[(i)] $h(x,\cdot)$ is rank 1-convex for a.e. $x\in {U }$;
    \item[(ii)]  $h(x,\Sigma)\leq \beta(1+\psi(\Sigma))$ for a.e. $x\in {U }$ and for every $\Sigma\in {\bf M}^{d\times N }$;
    \item[(iii)]  for all piecewise affine function $u\in W^{1,\infty}({U },\R^{ N})$  and all $V\in\A({U })$ \begin{equation}\label{uguaPA}
    H(u,V)= \int_V h(x,D u(x))\,\mathrm{d}x\;
    \end{equation}
 \item[(iv)]  for all $u\in W_{\psi}({U },\R^{ N})$ and all $V\in\A({U })$ \begin{equation}\label{upperest}
    H(u,V)\leq \int_V h(x,D u(x))\,\mathrm{d}x\;
    \end{equation}
\end{enumerate}

Moreover, if $H$ satisfies \textnormal{(H2)}, then 
for all $u\in W_{\psi}({U },\R^{N})$ and all $V\in\A({U })$ \begin{equation}\label{repres}
    H(u,V)= \int_V h(x,D u(x))\,\mathrm{d}x.
        \end{equation}
\end{thm}
\begin{proof} The proof follows the lines of  \cite[Proposition 9.2]{BDF} with some suitable changes. Accordingly, it  will be achieved in several steps.

\medskip\noindent{{Step 1:} \sl definition of $h$.} Fix $\Sigma\in {\bf M}^{N\times d}.$ By (H2)', $H(u_\Sigma , \cdot)$ can be extended to a Borel measure on ${U }$ which is absolutely continuous with respect to the Lebesgue measure since $H$ satisfies  (H3). Then there exists a density function $h_{\Sigma}\in L^1({U })$ defined by 
\[ h_\Sigma(x)=\limsup_{r\to 0^+}\frac{H(u_\Sigma , B_{r}(x))}{|B_{r}(x)|} \qquad \forall x\in {U }\]
such that
\[ H(u_\Sigma ,B)=\int_{B}g_{\Sigma}(x) \mathrm{d}x \qquad \forall B\in \mathcal B({U }).\]
We define $$h(x,\Sigma)=h_{\Sigma}(x) \qquad \forall (x,\Sigma)\in{U }\times {\bf M}^{N\times d}. $$
Thanks to (H3), we have that $h$ satisfies property (ii).

\medskip\noindent{{Step 2:} \sl integral representation on piecewise affine functions.} It is sufficient to repeat verbatim the proof of the analogue step in the proof of  \cite[Theorem 9.1]{BDF} to show that, thanks to (H1), (H2)', (H4) and to the Step 1,  the representation formula \eqref{repres} holds for every piecewise affine function $u\in W^{1,\infty}({U },\R^{ N})$.

\medskip\noindent{{Step 3:} \sl rank-$1$ convexity of $h$}. Following the proof of  Step 3 in   \cite[Theorem 9.1]{BDF}, given $\Sigma, Z \in {\bf M}^{N\times d}$ and  set $V(x)=(t \Sigma+(1-t)Z) x$, it is possibile to construct a sequence of piecewise affine  functions $V_n$ such that $DV_n\in \{\Sigma, Z \}$  and $V_n\weakst V$ in $W^{1,\infty}({U }, \R^{N}) $. 
In particular $\{V_n\}$ is bounded in $L^1( B_r(x_0),\R^N)$, thus we get that  the sequence $\{V_n\}$ is bounded in $W_{\psi}( B_r(x_0),\R^N)$. By Theorem \ref{musielakbanach}, we obtain that $V_n\rightharpoonup V$ in $W_{\psi}( B_r(x_0),\R^N)$.
Hence, proceeding as in the proof of the \cite[Theorem 9.1]{BDF} and by using (H5), we get that for every $ x_0\in {U }$ and every $0< r<d(x_0,\partial {U })$ it holds
\[ H(V, B_r(x_0))\leq t H(u_\Sigma , B_r(x_0))+(1-t) H(u_Z , B_r(x_0)) .\]
On dividing both sides by $|B_r(x_0)|$ and passing to the limsup as $r\to 0^+$, we get that $h(x_0,\cdot)$ satisfies (i). In particular, for every $x\in {U }$, we have that $h(x,\cdot)$ is locally Lipschitz, and hence
$h$ is a Carath\'eodory function.

\medskip\noindent{{Step 4:} \sl   a ''continuity'' property.} For every $V\in \mathcal A(U)$ and $u\in W_{\psi}(U,\mathbb R^N)$,  we define the functional $$\widetilde{H}(u,A)=\int_A h(x,D u(x))\,\mathrm{d}x.$$
Now we claim  that, if a sequence  $\{u_n\}\subseteq W^{1,1}(U,\mathbb R^N) $ converges to $u$ in  $W^{1,1}(A,\R^N)$, as $n\to +\infty$,  and satisfies 
\[\int_{A}  \psi(Du_n) \mathrm{d}x \to  \int_{A} \psi(Du) \mathrm{d}x<+\infty\,
\]
then  $\widetilde{H}(u_n,A)\to \widetilde{H}(u,A)$ as  $n\to +\infty$.  Indeed, let $\{u_n\}$ be as above. Up to a subsequence,
we can assume that  $D u_n(x)\to D u(x)$ for a.e. $x\in A$.
Since $h(x,\cdot)$ is continuous, we have that, as $n\to +\infty$, 
$h(x,D u_n(x))\to h(x,D u(x) )$ for a.e. $x\in A$.
Moreover, thanks to the growth condition (ii) satisfied by $h$, we have that  $$0\leq h(x,D u_n(x))\leq \beta(1+\psi(D u_n))   .$$ Hence it is sufficient to use Vitali convergence theorem, to  obtain the claimed convergence.

\medskip\noindent{{Step 5:} \sl   an inequality by continuity.}
Thanks to  Step 2,  for all $  A\in \mathcal{A}({U })$,  we have that  \[H(u,A)=\widetilde{H} (u,A)
 \  \hbox{for every piecewise affine function } u\in W^{1,\infty}({U },\R^N), \] 
 i.e. \eqref{uguaPA} holds. 
In order to prove \eqref{upperest}, let  $ u\in W_{\psi}( {U },\R^N)$,  then,  using Corollary \ref{PAdensity}, we have that 
 there exists a sequence $\{u_n\}$ of piecewise affine functions such that,  $u_n\to u$ in $W^{1,1}({U },\R^N)$  as $n\to +\infty$,  and 
$$
\int_{{U }}\psi (D u_n)\, \mathrm{d}x\to \int_{U } {\psi} (D u)\, \mathrm{d}x \quad \hbox{ as $n\to +\infty$}.
$$
 
Moreover, since  the sequence $\{u_n\}$ is bounded in $W_{\psi}({U },\R^N)$ which is a reflexive space (see Theorem \ref{musielakbanach}) and converges to $u$ in $W^{1,1}({U },\R^N)$,   we get that the sequence  $\{u_n\}$ weakly converges to $u$ in $W_{\psi}({U },\R^N)$, as $n\to +\infty$.
Passing to the limit in the equality
$H(u_n,A)=\widetilde{H} (u_n,A) $ and  exploiting the lower semicontinuity  assumption (H5) on $H$ and  the continuity property of $\widetilde H$ (provided in Step 4), 
 we get that the  inequality \eqref{upperest} holds.

\medskip\noindent{{Step 5:} \sl  equality by translation}. Finally, if ${H}$ satisfies the stronger assumption (H2), then
using  (H4),  it is sufficient to reason as in step 5 of the proof of  \cite[Theorem 9.1]{BDF} to conclude that 
\[H(u,A)= \widetilde{H}(u,A) \quad \forall v\in W_{\psi}( {U },\R^N),\quad  \forall A\in \mathcal{A}({U }).\] \end{proof}
\subsection{The general case}\label{gencase}
Throughout this section, a radial Carath\'eodory integrand $g\colon \R^d \times  {\bf M}^{N\times d} \to [0, +\infty)$ and $1<p^-\le p^+<+\infty$ satisfying the assumptions of Section \ref{convex homog} are fixed. Let $f\colon {U }\times{\bf M}^{N\times d} \to [0,+\infty)$ be  a Carathéodory function with the following property:  there exist $\alpha,\beta\in (0,+\infty)$  
such that for a.e.  $x\in{U }$ and for every $\Sigma\in {\bf M}^{N\times d}$  it holds
 \begin{equation}\label{growth-f}
\alpha\, g(x, \Sigma)\le f(x,\Sigma) \le \beta\left(1 +g(x, \Sigma)\right).
 \end{equation}
Observe that, by \eqref{growth-g}, up to redefining the constants $\alpha$ and $\beta$  we also have 
\begin{equation*}\label{growthgeneral-p-p+}
 \alpha\left( |\Sigma|^{p^-}-1\right) \le f(x,\Sigma) \le \beta\left(2 +|\Sigma|^{p^+}\right).
\end{equation*}
For every $\varepsilon>0$, we introduce $F_{\varepsilon}:L^1({U })\times \mathcal{A}({U })\to [0,+\infty]$ defined by
\begin{equation}\label{casoscalare}
F_{\varepsilon}(u, V) = 
\begin{cases} 
\displaystyle \int_V f\left(\frac{x}{\varepsilon}, Du(x)\right) \, \mathrm{d}x\, & \text{if } u \in W^{1,1}(V; \mathbb R^{N}), \\
+\infty & \text{otherwise},
\end{cases}
\end{equation}
and set $F_{\varepsilon}(u):=F_{\varepsilon}(u, {U }).$
Given a sequence ${\varepsilon_k}\to 0$, for every $V\in \mathcal A({U })$ we 
 define  
\begin{align*}
    F'(u,V)=\Gamma(L^1)\text{-}\liminf_{k\to +\infty} F_{\varepsilon_k}(u,V), \\
F''(u,V)=\Gamma(L^1)\text{-}\limsup_{k\to +\infty} F_{\varepsilon_k}(u,V).
\end{align*}
Notice that $F'$ and $F''$ may depend on the chosen subsequence $\{\varepsilon_k\}_k$. Throughout this section, we consider a countable subset $\mathcal D\subseteq \A({U })$ such that ${U } \in \mathcal D$, $\mathcal D$ is dense in $\A({U })$ and which is stable under finite union, and  a sequence  ${\varepsilon_k}\to 0$ is fixed from the beginning with the following property:  
\begin{equation}\label{subsequenza}
     F'(\cdot,D)=F''(\cdot, D) \quad \hbox{ for every } D \in \mathcal D.
 \end{equation}
As the considered subset of $\A({U })$ is countable, this can be ensured by compactness of the $\Gamma$-convergence and a diagonal argument. The main result of this section is the following.
\begin{thm}\label{mainthmnonconvex} 
Let  $U\subseteq \mathbb R^d$ be a bounded connected Lipschitz open set and, for $f$ as in \eqref{growth-f}, let    $F_\varepsilon$  be given by \eqref{casoscalare}. Let  $\gamma: {\bf M}^{N\times d}\to [0,+\infty)$  be the $N$-function given by \eqref{fhom} and satisfying  the reinforced-$\Delta_2$ property \eqref{Delta2gamma}.
Fix a sequence ${\varepsilon_k}\to 0$ such that \eqref{subsequenza} holds. Then, there exists a Carath\'eodory function  $\eta\colon U\times {\bf M}^{N\times d}\to [0,+\infty)$  such that $\eta(x, \cdot)$ is rank-$1$ convex for a.e. $x \in U$ and the sequence $F_{\varepsilon_k}$ 
$\Gamma$-converges, with respect to the $L^1(U,\R^N)$ norm, to the functional 
\begin{equation*}
F(u):= \begin{cases} 
\displaystyle\int_{U} \eta(x, D u)\,\mathrm{d}x& \hbox{if }u\in W_{\gamma}(U,\R^N), \\
+\infty & \hbox{otherwise}\,.
\end{cases}
\end{equation*}

\end{thm}
Before approaching the proof of this result, we shortly review the role of the periodicity of $g(\cdot, \Sigma)$ in its derivation.

\begin{rem}\label{no-periodicity2}
We observe that for the results in this section we do not explicitly need to require  $f(\cdot,\Sigma)$ to  be periodic, but periodicity is assumed in the growth function $g(\cdot, \Sigma)$. However, the latter is only needed to ensure that the corresponding family of integral functionals $G_\varepsilon$ defined in \eqref{casovett} have a $\Gamma$-limit which can be represented by a function $\gamma$ independent of $x$. Whenever this information is available (for instance, as a consequence of the more general stationarity assumption considered in Section \ref{stochastic}), the above result can be proved without any  periodicity assumption on $f$.
\end{rem}

The proof strategy towards Theorem \ref{mainthmnonconvex} relies on Theorem \ref{repr}. In order to invoke the latter, in the following we prove that $F'$ satisfies hypotheses (H1)-(H5). Similarly to the proof of Lemma \ref{limsup}, where the hypotheses of the representation theorem \cite[Theorem 9.1]{BDF} in $W^{1,p^+}({U },\R^N)$ were checked, the only nontrivial assumptions we need to  verify are (H2) and (H3).
We follow the steps of the preceding section. Firstly,  the uniform fundamental estimate given by Lemma \ref{fundest} 
still holds with  the functionals $F_{\varepsilon}$, given  by \eqref{casoscalare},  in place of $G_{\varepsilon}$,   thanks to \cite[Proposition 21.7]{BDF}. Moreover, the first part of the following lemma  provides property (H3): here the function $\gamma$ is the same given  by Theorem \ref{gammamodel}.

\begin{lemma} \label{bounds} Let $V\in \mathcal A({U })$. Then  
\begin{enumerate}
\item  for every $u\in W_{\gamma}(U , \R^N)$ it holds
\begin{align}\label{alto}
        F''(u,V)\le  \beta\int_V(1+\gamma(D u))\,\mathrm{d}x;
    \end{align}
 \item for every $u\in L^1({U },\R^N)$,  if $V\in \mathcal A_0({U })$ and $F'(u,{V})<+\infty$  
     then it holds   \begin{align}\label{basso}
        \alpha\int_V\gamma(D u)\,\mathrm{d}x\le F'(u,V).
    \end{align}
       In particular $u\in W_\gamma({V},\R^N)$;
 \item if $u\in W_\gamma({U },\R^N)$ then \eqref{basso} holds without  any additional assumption on $V$.
\end{enumerate}
\end{lemma}

\begin{proof}
Let $G_{\varepsilon}$ be given  by \eqref{casovett}. 
In order to show \eqref{alto}, let $u\in W_\gamma({U },\R^N)$. 
Then, thanks to  \eqref{eq: local},
for the fixed sequence  $\varepsilon_k\to 0$ such that \eqref{subsequenza} holds,  there exists a sequence $\{u_k\}\subseteq L^1(U,\R^N)$, $u_k\to u$ in $L^1(V,\R^N)$, as $k\to +\infty$, such that 
 \[\lim_{k\to +\infty} G_{{\varepsilon_k}}(u_k,V)=\int_V \gamma(D u)\,\mathrm{d}x<+\infty.\]
In particular, there exists   $k_0\in \N$ such that  $u_k\in  W^{1,1}(V,\R^N)$ for $k\geq k_0$ and, using the growth condition \eqref{growth-f} from above, we get 
\[
\begin{split}
        F''(u,V)&\le \limsup_{k\to +\infty} F_{{\varepsilon_k}}(u_k,V)\leq \limsup_{k\to +\infty} \int_V \beta \left(1+g\left(\frac{x}{{\varepsilon_k}}, Du_k\right)\right)\, \mathrm{d}x=\beta\int_V(1+\gamma(D u))\,\mathrm{d}x\,.
    \end{split}
    \]
    In order to prove \eqref{basso} let  $\{w_k\}\subseteq L^1(U,\R^N)$ be a recovery sequence   for   $F'(u,V)<+\infty$. Then, there exists $k_0\in \N$ such that $w_k\in  W^{1,1}(V,\R^N)$ for $k\geq k_0$.   By using the growth condition \eqref{growth-f} from below and  Theorem \ref{gammamodel} (applied with the Lipschitz open set $V$ in place of $U$),  we get that \begin{align*}
        \alpha\int_V\gamma(D u)\,\mathrm{d}x=\alpha G'(u,V) \le \alpha\liminf_{k\to +\infty} G_{\varepsilon_k}(w_k, V)\le \liminf_{k\to +\infty} F_{\varepsilon_k}(w_k,V)=F'(u, V)<+\infty
    \end{align*}
which gives the desired conclusion. Part (3) follows from \eqref{eq: local} in a  similar way.
\end{proof}

Now we will turn our attention to (H2). As  in the previous section, we will exploit   De Giorgi-Letta criterion. To this aim, we provide two results that are analogous to Lemmas \ref{subadditivity} and \ref{innreg}. We note that the functional $F_{\varepsilon_k}$ does not satisfy \eqref{eq:dx}: hence  we need to provide some  necessary  modifications    to the proof of Lemma \ref{subadditivity}.

\begin{lemma}\label{subadditivityF}
  Let  $u\in L^1({U },\R^N)$. Then $F'(u,\cdot)$ is superadditive on $\mathcal A({U })$.
Moreover,  for every $V,V',W\in \mathcal A({U })$ and  $V'\Subset V$ we have:
    \begin{align}
        F'(u,V'\cup W)\le F'(u,V)+F''(u,W).\label{eq: subaddF}
    \end{align}
\end{lemma}

\begin{proof}
 The proof of superadditivity of $F'$  is the same as in Lemma \ref{subadditivity}. Now  we prove \eqref{eq: subaddF}.   
We fix two sequences $u_k$ and $v_k$ such that $u_k \to u$ in $L^1(V, \R^N)$, $v_k \to u$ in $L^1(W, \R^N)$ and 
\[
\liminf_{k\to +\infty}F_{\varepsilon_k}(u_k, V)=F'(u, V) \quad \mbox{ and }\limsup_{k\to +\infty}F_{\varepsilon_k}(v_k, W)=F''( u, W).
\]
We also fix a subsequence $k_\ell$ such that the liminf above holds as a limit: along each further subsequence $k_{\ell_m}$ thereof we then have
\begin{equation}\label{eq: estratte}
\lim_{m\to +\infty}F_{\varepsilon_{k_{\ell_m}}}(u_{k_{\ell_m}}, V)=F'(u, V) \quad \mbox{ and }\limsup_{m\to +\infty}F_{\varepsilon_{k_{\ell_m}}}(v_{k_{\ell_m}}, W)\le F''( u, W),
\end{equation}
as the limsup is possibly decreased when taking a subsequence.

Since we can assume that $F'(u,V)$ and $F''(u,W)$ are finite (otherwise, the thesis is trivial),  exploiting the growth condition from below  given by \eqref{growth-f}, there exists $L<+\infty$ such that for every $k\in\N$ 
\begin{equation}\label{stima15}
    \alpha \left(\int_V g\left(\frac{x}{{\varepsilon_k}}, Du_k\right)\, \mathrm{d}x + \int_Wg\left(\frac{x}{{\varepsilon_k}}, Dv_k\right)\, \mathrm{d}x\right) \leq L.
\end{equation}
Moreover, since $u_M\to u$ in $L^1({U }, \mathbb R^N)$, thanks to Lemma \ref{tronc}, by lower semicontinuity, we have that  $$F'(u,V'\cup W)\leq \liminf_{M\to +\infty} F'\left(u^M,V'\cup W\right).$$
For a fixed $\sigma>0$, let  $M\in\N$ such that $\frac{\beta(|V|+|W|+L \alpha^{-1})}{M+1}<\sigma$ and it holds 
\begin{equation}\label{sciF'}
F'(u,V'\cup W)\le F'\left(u^K,V'\cup W\right) + \sigma \qquad \hbox{for every } K> M.
\end{equation}
We can also assume that $M$ is chosen so that
\begin{equation}\label{superlevels}
\frac\beta M\left(\|u\|_{L^1(V, \R^N)}+\|u\|_{L^1(W, \R^N)}\right)\le \sigma\,.
\end{equation}
From \eqref{stima15} we have 
\begin{equation*}
\begin{split}
\sum_{i=M}^{2M} \beta \Bigg(
    &\int_{V\cap \{ 2^i < |u_k| \le 2^{i+1} \}}
        \left(1 + g\left(\frac{x}{\varepsilon_k}, Du_k\right)\right)\, \mathrm{d}x\, +
    \\
    &
    \beta\int_{W\cap \{ 2^i < |v_k| \le 2^{i+1} \}}
        \left(1 + g\left(\frac{x}{\varepsilon_k}, Dv_k\right)\right)\, \mathrm{d}x
\Bigg)
\leq \beta(|V|+|W|+L \alpha^{-1})<\sigma(M+1).
\end{split}
\end{equation*}
Hence, for each $k\in\N$ there exists $j({\varepsilon_k})\in\N$, $M \leq j({\varepsilon_k})\leq 2M$,  such that 
\begin{equation}\label{troncJ}
\begin{split}
&\beta  \int_{V\cap \{2^{j(\varepsilon_k)} < |u_k| \le 2^{j(\varepsilon_k)+1}\}}
        \left(1 + g\left(\frac{x}{\varepsilon_k}, Du_k\right)\right)\, \mathrm{d}x \, +
    \\
    &  
    \beta\int_{W\cap \{2^{j(\varepsilon_k)} < |v_k| \le 2^{j(\varepsilon_k)+1}\}}
        \left(1 + g\left(\frac{x}{\varepsilon_k}, Dv_k\right)\right)\, \mathrm{d}x
< \sigma .
\end{split}
\end{equation}
Since $\{j({\varepsilon_k})\}_{k}$ is a bounded sequence of integers, there exists $J\in \mathbb N$, $M\leq J\leq 2M$, and a subsequence $\{\varepsilon_{k_{\ell_m}}\}_m$, where $k_\ell$ is fixed before \eqref{eq: estratte}, such that $j({\varepsilon_{k_{\ell_m}}})=J$ for every $m\in \mathbb N$. We fix the value $K=2^J$.  For every $k\in \N$, we split  the set $V$and $W$  as follows \[V=  \{x\in V: |u_k|\le 2^{J}\}\cup  \{x\in V: 2^{J}<|u_k|\le 2^{J+1}\} \cup  \{x\in V: |u_{k}|\ge 2^{J+1}\}=V^1_{k}\cup V^2_{k}\cup V^3_{k},\]
\[W=  \{x\in W: |v_k|\le 2^{J}\}\cup  \{x\in V: 2^{J}<|v_k|\le 2^{J+1}\} \cup  \{x\in V: |v_{k}|\ge 2^{J+1}\}=W^1_{k}\cup W^2_{k}\cup W^3_{k}.\]   We observe  that, being  $2^{J+1}>K$, we have that $Du_k^K=0$ for a.e  $x\in V^3_{k}$ and $Dv_k^K=0$  for a.e  $x\in W^3_{k}$. Since, for a.e.\ $x\in{U }$,   $g(x,O)=0$ gives $f(x,O)\le \beta$, thanks to \eqref{growth-f}, by Chebyshev inequality and  taking into account that  $2^{J+1} >M$, for all $k\in \N$ we get that 
\begin{equation}\label{eq: ovvia}
\begin{split}
\int_{V^3_{k}} f\!\left(\frac{x}{\varepsilon_k},Du_k^{K}\right)\, \mathrm{d}x +\int_{W^3_{k}}
    f\!\left(\frac{x}{\varepsilon_k},Dv_k^{K}\right)\, \mathrm{d}x  &=   \int_{V^3_{k}}
    f\!\left(\frac{x}{\varepsilon_k},O\right)\, \mathrm{d}x +\int_{W^3_{k}}
    f\!\left(\frac{x}{\varepsilon_k},O\right)\, \mathrm{d}x\\
  \le  \vphantom{\int_{W^3_{k}}
    f\!\left(\frac{x}{\varepsilon_k},O\right)\, \mathrm{d}x }&\frac\beta M\left(\|u_k\|_{L^1(V, \R^N)}+\|v_k\|_{L^1(W, \R^N)}\right)\,.
\end{split}
\end{equation}
Now, along the subsequence ${k_{\ell_m}}$, since $g(x,\cdot)$ is radial increasing and $|Du^K_{k_{\ell_m}}|\leq |Du_{k_{\ell_m}}|$, we have that
\begin{equation}\label{troncJ2}
 \int_{V^2_{k_{\ell_m} }}
        \left(1 + g\left(\frac{x}{\varepsilon_k}, Du^K_{k_{\ell_m}}\right)\right)\, \mathrm{d}x \leq  \int_{V^2_{k_{\ell_m} }}        \left(1 + g\left(\frac{x}{\varepsilon_k}, Du_{k_{\ell_m}}\right)\right)\, \mathrm{d}x.
\end{equation}
Taking into account that    $u ^K_{k_{\ell_m}}=u_{k_{\ell_m}}$  on $V^1_{k_{\ell_m} }$,  applying on $V^2_{k_{\ell_m}}$  the growth condition from above \eqref{growth-f} combined with  the estimate \eqref{troncJ2} and \eqref{troncJ}   with $j({\varepsilon_{k_{\ell_m}}})=J$, and  using \eqref{eq: ovvia} to estimate the integral on $V^3_{k_{\ell_m}}$ (and proceeding in  the same way on $W$),  we obtain
\begin{equation}\label{stima2}
\begin{split}
 F_{\varepsilon_{k_{\ell_m}}}(u_{k_{\ell_m}}^{K},V) & + F_{\varepsilon_{k_{\ell_m}}}(v_{k_{\ell_m}}^{K},W) \le\\
F_{\varepsilon_{k_{\ell_m}}}(u_{k_{\ell_m}},V) + F_{\varepsilon_{k_{\ell_m}}}(v_{k_{\ell_m}}, W)&+ \sigma + \frac\beta M\left(\|u_{k_{\ell_m}}\|_{L^1(V, \R^N)}+\|v_{k_{\ell_m}}\|_{L^1(W, \R^N)}\right)\,.
\end{split}
\end{equation}
Using the fundamental estimate in Lemma \ref{fundest}, we construct a sequence $w_k=\varphi_{_k} u^K_k + (1 - \varphi_{_k})v^K_k$ such that
\[
    F_{\varepsilon_k}(w_k,V'\cup W)\le (1+\sigma)\bigl[F_{\varepsilon_k}(u_k^K,V)+F_{\varepsilon_k}(v_k^K,W)\bigr]+M_\sigma\int_{V\cap W} |u_k^K-v_k^K|^{p^+}\, \mathrm{d}x + \sigma\,.
 \]
Since  $\|u^{K}_k - v^{K}_k\|_{L^{p^+}(V\cap W,\mathbb R^N)}\to 0 \ $ as $k\to +\infty$, this gives
\begin{equation}\label{feF}
 \liminf_{k\to+\infty}F_{\varepsilon_k}(w_k,V'\cup W)\le (1+\sigma)\liminf_{k\to+\infty}\bigl[F_{\varepsilon_k}(u_k^K,V)+F_{\varepsilon_k}(v_k^K,W)\bigr]+\sigma.
\end{equation}
Since  $u_k \to u$ in $L^1(V, \R^N)$, $v_k \to v$ in $L^1(W, \R^N)$, and $w_k\to u^{K}$ in $L^1(V'\cup W, \mathbb R^N)$,  by exploiting \eqref{sciF'}, \eqref{feF}, \eqref{stima2},  \eqref{superlevels} and  \eqref{eq: estratte}, we eventually get 
\begin{equation}
\begin{aligned}
    F'(u,V'\cup W)
    &\le F'\left(u^{K},V'\cup W\right) + \sigma \\
    &\le \liminf_{k\to +\infty} F_{\varepsilon_k}(w_k,V'\cup W) + \sigma \\
    &\le (1+\sigma)\left[\liminf_{k\to +\infty} \left(F_{\varepsilon_k}(u_k^{K},V)
        + F_{\varepsilon_k}(v_k^{K},W)\right)\right] + 2\sigma \\
    &\le  (1+\sigma)\left[\liminf_{m\to +\infty} \left(F_{\varepsilon_{k_{\ell_m}}}(u_{k_{\ell_m}}^{K},V) + F_{\varepsilon_{k_{\ell_m}}}(v_{k_{\ell_m}}^{K},W)\right)\right] + 2\sigma\\
    &\le (1+\sigma)\left[\liminf_{m\to +\infty} F_{\varepsilon_{k_{\ell_m}}}(u_{k_{\ell_m}}, V) + \limsup_{m\to +\infty} F_{\varepsilon_{k_{\ell_m}}}(v_{k_{\ell_m}},W)+2\sigma\right]+2\sigma \\
    &\le (1+\sigma)\left[F'(u,V)+F''(u,W)+2\sigma\right]+2\sigma.
\end{aligned}
\end{equation}
As $\sigma\to 0$, we obtain the thesis \eqref{eq: subaddF}.
\end{proof}

The previous result can  be used to infer the inner regularity and the subadditivity of $F'(u,\cdot)$ that we discuss in the next Lemma.
\begin{lemma}\label{innregF} For every $u\in W_{\gamma}({U },\R^N)$ we have that   $F'(u,\cdot)$ is an   inner regular  increasing set functions on $\mathcal A( {{U }})$. Moreover $F'(u,\cdot)$ is subadditive and superadditive for every $u\in W_{\gamma}({U },\R^N)$.
\end{lemma}
\begin{proof}
 Fix $Z\in \mathcal{A}({U })$. Let $K$ be a compact set of $Z$  and choose $Z', Z''\in \mathcal{A}({Z })$  such that $K\subset Z''\Subset Z' \Subset Z$.  Set $W=Z\setminus K$. We apply Lemma \ref{subadditivityF} with  the triple $Z''$, $Z'$ and $W=Z\setminus K$, and observe that $Z''\cup W=Z$. Also using \eqref{alto} on $W$,  we then have
\begin{align*}
    F'(u,Z)&\le F'(u,Z')+F''(u,Z\setminus K)\le \sup\{F'(u,Z'):\, Z'\Subset Z\}+F''(u,Z\setminus K)\\
    &\le \sup\{F'( u,Z'):\, Z'\Subset Z\}+\beta\displaystyle \int_{Z\setminus K} (1+\gamma(Du))\, \mathrm{d}x.
\end{align*}
Then, as $|Z\setminus K|\to 0$, we get 
\begin{equation*}
    F'(u,Z)\le \sup\{F'(u,Z'):\, Z'\Subset Z\}\,.
\end{equation*}
Since the reverse inequality is trivial, we obtain that $F'(u,\cdot)$ is inner regular. 
To prove the subadditivity, given $V, W \in \mathcal A(U)$, choose $V'\Subset V$ and $W' \Subset W$ with $V', W' \in \mathcal D$ and use \eqref{eq: subaddF} and \eqref{subsequenza} to get
\[
F'(u, V'\cup W')\le F'(u, V)+F''(u, W')= F'(u, V)+F'(u, W')\le F'(u, V)+F'(u, W).
\]
Passing to the supremum on $V'\Subset V$ and $W' \Subset W$,   \cite[Lemma 14.20]{DM} and the inner regularity of $F'$ give the desired conclusion.
\end{proof}
We are now in a position to represent $F'$ as in integral functional on  $W_{\gamma}({U },\R^N)$.
\begin{lemma}\label{limsupF}
 Under the assumptions of Theorem \ref{mainthmnonconvex}, there exists a Carathéodory function $\eta\colon {U }\times {\bf M}^{d\times N }\to [0,+\infty)$ such that  
\begin{enumerate}
    \item for every  $u\in W_{\gamma}({U },\R^N)$ and for every  $V\in \mathcal{A}({U })$  it holds
    \begin{equation}\label{reprlimsupF}
    F'(u,V)=  \displaystyle \int_V \eta (x, D u)\,\mathrm{d}x;
    \end{equation}
    \item $\alpha\gamma(\Sigma)\leq \eta(x,\Sigma)\leq \beta\left(2+\gamma(\Sigma)\right)$ for a.e. $x\in {U }$ and for every $\Sigma\in {\bf M}^{N\times d}$;
    \item $\eta(x,\cdot)$ is rank-1 convex for a.e. $x\in {U }$.
 
  \end{enumerate}
\end{lemma}
\begin{proof} Proposition \ref{gammaN} ensures  that the function  $\gamma$ given by  \eqref{fhom} is an $N$-function satisfying  the reinforced-$\Delta_2$ property. Thanks to the compact embedding of $W_{\gamma}({U },\R^N)$  in $L^1({U },\R^N)$ (see Theorem \ref{musielakbanach}), for any \( V \in \mathcal{A}({U }) \),    the functional   \( F'(\cdot, V) \),  which is  sequentially lower semicontinuous with respect to the $L^1(V)$-convergence,  is sequentially  lower semicontinuous with respect to the weak convergence of $W_{\gamma}({U }, \R^N)$.
  Then, taking into account also Lemmas \ref{bounds} Part (1), \ref{subadditivityF} and \ref{innregF}, we get that $F'$ satisfies all the assumptions of  Theorem \ref{repr} with $\gamma$ in place of $\psi$. The application of the latter, together with the lower bound 
\eqref{basso} to get  the growth condition from below for $\eta$, entails the desired result.
\end{proof}
As in Section \ref{convex homog}, for the applications to the stochastic homogenization, we prove an analog of Proposition \ref{G''min}.
We note that  formula \eqref{eq:cellformula1F} below, when applied to $F'=H$,   extends \cite[Theorem 2]{bflm} to  the case of translation invariant functionals $H$, defined on an  anisotropic Musielak-Orlicz Space $W_\psi(U,\R^N)$, under the non standard growth condition 
\[
       \alpha\int_V\psi(D u)\,\mathrm{d}x\le H(u,V)\le  \beta\int_V(1+\psi(D u))\,\mathrm{d}x
   \qquad \forall u\in W_\psi({U },\R^N)\,, \  \forall\, V\in \mathcal{A}(U).\]

\begin{prop}\label{minF''eta}
Let $m_{F'}(u_\Sigma, Q_r(x))$ and  $m_{F_{\varepsilon}}(u_\Sigma, Q_r(x))$ be given by  $\eqref{eq: minimointornobordo}$. Under the assumptions of this section, the following holds.
   \begin{enumerate}
\item 
For every $\Sigma \in \mathbf{M}^{N\times d}$ and for a.e. $x \in U $ we have
\begin{equation}\label{eq:cellformula1F}
\eta(x,\Sigma)
=
\lim_{r\to 0^+}
\frac{m_{F'}(u_\Sigma, Q_r(x))}{|Q_r(x)|}.
\end{equation}

\item 
For every $\Sigma \in \mathbf{M}^{N\times d}$ and for  for a.e.   $x \in U $, we have
\begin{equation}\label{eq:convminimaF}
\limsup_{r\to 0^+}
\liminf_{k\to +\infty}
\frac{m_{F_{\varepsilon_k}}(u_\Sigma, Q_r(x))}{|Q_r(x)|}\le \eta(x, \Sigma)\le \liminf_{\theta \to 1^-}\liminf_{r\to 0^+}
\liminf_{k\to +\infty}
\frac{m_{F_{\varepsilon_k}}(u_\Sigma, Q_{\theta r}(x))}{|Q_{\theta r}(x)|}.
\end{equation}
\end{enumerate}

\end{prop}
\begin{proof} Observe that, if  $F'(v, Q_r(x))<+\infty$, then, thanks to Lemma \ref{bounds}, Part (2),  $v \in W_\gamma(Q_r(x), \R^N)$, and, by the condition $v=u_\Sigma$ on a neighborhood of $\partial Q_r(x)$, we have that each competitor for $m_{F'}(u_\Sigma, Q_r(x))$ can be extended equal to $u_{\Sigma}$  outside  $Q_{r}(x)$  so that  such a function belongs to $W_\gamma(U,\R^N)$. 
 \noindent Now,  we note that, with the same proof of \eqref{mdelta}, it holds \begin{equation*}\label{mdeltaF'}
\lim_{\delta\to 0^+}
m_{F'}(u_\Sigma, Q_{r-\delta}(x))
=
m_{F'}(u_\Sigma, Q_r(x)).
\end{equation*} 
 Exploiting the measure and coercivity properties of $F'$, the proof of \eqref{eq:cellformula1F} can now be carried out as in \cite[Lemma 3.3 and 3.5]{bfm}.
The proof of the second inequality in \eqref{eq:convminimaF} can be obtained along similar lines  in  Proposition \ref{G''min}, up to replacing limsup with liminf where needed we note that, thanks to the properties satisfied by   $F'$, we do not need to show any additional property as that claimed  in \eqref{eq:uguali}). Some relevant differences appear instead for the first inequality, which we discuss in detail.
We fix $\sigma>0$ and let $v$ be such that $v=u_\Sigma$ in a neighborhood of $\partial Q_r(x)$ and
\[
F'(v, Q_r(x) )\le m_{F'}(u_\Sigma, Q _r(x))+\sigma\,<+\infty.
\] 
Then $v$ can can be extended to a function in $W_\gamma(U,\R^N)$.  We consider  a recovery sequence $\{v_k\}$ for $F'(v, Q_r(x))$ and a subsequence $\{k_\ell\}$ such that \[
\lim_{\ell \to +\infty} F_{\varepsilon_{k_\ell}}(v_{k_\ell}, Q_r(x))= F'(v, Q_r(x))\,. \]
Now, choose a sufficiently large integer $M$ such that $M\geq 
\beta \sigma^{-1} \|v\|_{L^1(Q_r(x), \R^N)}.
$
Arguing as in \eqref{stima2}, given the sequence $v_{k_\ell}$ we can find a  subsequence and an integer $J\in [2^M,4^M]$ such that 
 \[
\liminf_{m \to +\infty}
F_{\varepsilon_{k_{\ell_m}}}\bigl(v_{k_{\ell_m}}^{J}, Q_r(x)\bigr)
\;\le\;
\lim_{m \to +\infty}
F_{\varepsilon_{k_{\ell_m}}}\bigl(v_{k_{\ell_m}}, Q_r(x)\bigr)
+ 2\sigma
= F'(v, Q_r(x)) + 2\sigma.
\]
    Fixed $\delta>0$, set $z_k=\varphi_kv_k^{J}+(1-\varphi_k)v^{J}$, where $\varphi_k$ is a suitable cutoff function between $Q_{r-\delta}(x)$ and $Q_r(x)$ such that the following  fundamental estimate holds: 
\[
\begin{split}
F_{\varepsilon_k}(z_k, Q_{r}(x)) &\leq (1 + \sigma)\left(F_{\varepsilon_k}(v^J_k, Q_r(x)) + F_{\varepsilon_k}(v, Q_r(x) \setminus \overline{Q_{r-\delta}(x)})\right) \\
&+ M_\sigma \int_{(Q_r(x) \setminus Q_{r-\delta}(x)} |v^J_k - v^J|^{p^+} \, \mathrm{d}x + \sigma. 
\end{split}
    \]
Since $z_k\to v^J$ in $L^1(Q_r(x),{\bf M}^{d\times N})$ and $z_k=u_\Sigma$ on a  neighborhood of   $\partial Q_r(x)$ (provided that $M$ was chosen greater than $\log_2(\|u_{\Sigma}\|_{L^{\infty}(U)}$), using the fundamental estimate \eqref{eq:lp_fundamental_estimate} for $F_\varepsilon$, the growth condition \eqref{alto}, and the previous inequality we obtain 
\[
\begin{split}
        \liminf_{k\to +\infty}m_{F_{\varepsilon_k}}(u_\Sigma, Q_r(x))&\le \liminf_{k\to +\infty} F_{\varepsilon_k}(z_k,Q_r(x))\\
&\le  ( 1+\sigma)\liminf_{m \to +\infty}
F_{\varepsilon_{k_{\ell_m}}}\bigl(v_{k_{\ell_m}}^{J}, Q_r(x)\bigr)+\sigma+\beta\int_{Q_r(x)\setminus Q_{r-\delta}(x)} \left(1+\gamma(Dv)\right)\,\mathrm{d}y \\
&=
( 1+\sigma) F'(v, Q_r(x))+ \sigma+\beta\int_{Q_r(x)\setminus Q_{r-\delta}(x)} \left(1+\gamma(Dv)\right)\,\mathrm{d}y \\
&\le ( 1+\sigma)(m_{F'}(u_\Sigma, Q_r(x))+\sigma)+\sigma+ \beta\int_{Q_r(x)\setminus Q_{r-\delta}(x)} \left(1+\gamma(Dv)\right)\,\mathrm{d}y\,.
\end{split}
    \]
Taking the limit as $\delta\to 0^+$ and as $\sigma\to 0^+$, we obtain 
\[
 \liminf_{k\to +\infty}m_{F_{\varepsilon_k}}(u_\Sigma, Q_r(x))\le m_{F'}(u_\Sigma, Q_r(x))\,; 
\]
on dividing by $|Q_r(x)|$ we get the thesis when $r \to 0^+$.
\end{proof}

We are now in a position to prove the main result of this section.
\begin{proof}[Proof of Theorem \ref{mainthmnonconvex}]
 Let  $\eta$ be the function provided by Lemma \ref{limsupF}. By \eqref{subsequenza} with $D=U$ we have
\begin{equation*}
F'(u,U)= F''(u, U)=\int_U\eta(x,Du)\,\mathrm{d}x
\end{equation*}
 for each $u\in W_{\gamma}({U },\R^N)$.
Taking into account that, by Lemma \ref{bounds}, Part (1) and  (3),   for every $u\in L^1({U },\R^N)$ it holds    \[ F'(u,U)<+\infty \Longleftrightarrow u\in W_{\gamma}({U },\R^N),\] 
this concludes the proof. 
\end{proof}
\section{Stochastic Homogenization}\label{stochastic}
Let $(\Omega,\mathcal F, \mathbb P)$ be a probability space and let $\tau=\{\tau_z\}_{z\in\Z^d}$ be a measure-preserving group action on $(\Omega,\mathcal F, \mathbb P)$. We will assume that $\tau$ is ergodic, that is, the probability of the $\tau$-invariant sets is 0 or 1.  We consider $g\colon \Omega\times \R^d \times \R \to [0, +\infty)$ and $f\colon \Omega\times \R^d\times{\bf M}^{N\times d} \to [0,+\infty)$, both measurable in their product space (endowed with its obvious $\sigma$-algebra), and $1<p^-\le p^+<+\infty$, not depending on $\omega$, satisfying the following assumptions:
\begin{enumerate}
\item for each $\omega\in\Omega$, $g(\omega, \cdot, \cdot)$ and $f(\omega,\cdot,\cdot)$ are Carathéodory functions;

\item for each $\omega\in\Omega$ and  a.e.\ $x \in \mathbb{R}^d$, $g(\omega, x,\cdot)$ is convex and increasing on $[0,+\infty)$;

\item \normalfont(inc)$_{p^-}$ the map
$s \in (0,+\infty) \mapsto \dfrac{g(\omega, x,s)}{s^{p^-}}$
is increasing  for each $\omega\in\Omega$ and  a.e.\ $x \in \mathbb{R}^d$;

\vspace{2pt}

\item \normalfont(dec)$_{p^+}$ the map
$s \in (0,+\infty) \mapsto \dfrac{g(\omega, x,s)}{s^{p^+}}$
is decreasing  for each $\omega\in\Omega$ and  a.e.\ $x \in \mathbb{R}^d$;

\item \normalfont(A0) there exists $\alpha, \beta > 0$ such that
\[
\alpha \le g(\omega, x,1) \le \beta
\]
for each $\omega\in\Omega$ and a.e.\ $x \in \mathbb{R}^d$;

 \item  for each $\omega\in\Omega$, for a.e.  $x\in\R^d$ and for every $\Sigma\in {\bf M}^{N\times d}$,   it holds
    \begin{equation}\label{growth-fomega}
\alpha g(\omega, x, |\Sigma|)\le f(\omega,x,\Sigma) \le \beta\left(1+g(\omega, x, |\Sigma|)\right);
    \end{equation}
    
 \item (stationarity)  $f$ and $g$ are stationary, that is, for a.e. $x\in \R^d$, it holds
   \begin{equation*}        f(\tau_z\omega,x,\Sigma)=f(\omega,x+z,\Sigma)\quad \hbox{ and }\quad  g(\tau_z\omega,x, |\Sigma|)=g(\omega,x+z,|\Sigma|)
\end{equation*}
for every $(\omega,z,\Sigma)\in \Omega\times \Z^d\times {\bf M}^{N\times d} $. 
\end{enumerate}

For the ease of notation we again write $g(\omega, x, \Sigma)$ in place of $g(\omega, x, |\Sigma|)$. As in Section \ref{nonconvexcase},  let $U\subseteq \mathbb R^d$ be a bounded Lipschitz  open set and  we consider the functionals $F_{\varepsilon}, G_\varepsilon:\Omega\times L^1(\R^d,\R^N)\times \mathcal{A}({ U})\to [0,+\infty]$ given by  
\begin{equation}\label{casostoc}
F_{\varepsilon}(\omega,u, V) = 
\begin{cases} 
\displaystyle \int_V f\left(\omega,\frac{x}{\varepsilon}, Du(x)\right) \, dx\,, & \text{if } u \in W^{1,1}(V; \R^N), \\
+\infty & \text{otherwise}
\end{cases}
\end{equation}
and
\begin{equation}\label{G_varepsilonomega}
    G_\varepsilon(\omega, u, V)=\begin{cases} 
\displaystyle \int_V g\left(\omega,\frac{x}{\varepsilon}, Du(x)\right) \, dx\,, & \text{if } u \in W^{1,1}(V; \R^N), \\
+\infty\, & \text{otherwise,}
\end{cases}
\end{equation}
respectively. Moreover, we  set $F_{\varepsilon}(\omega,u):=F_{\varepsilon}(\omega,u, { U})$ and   $G_{\varepsilon}(\omega,u):=G_{\varepsilon}(\omega,u, { U})$.
For every  $V\in \mathcal A_0(U)$ we  set \[\mu_\Sigma(\omega,V)=m_{F_1(\omega,\cdot)}(u_\Sigma,V),\]  where $F_1(\omega,\cdot)$ is the functional given by \eqref{casostoc} taking $\varepsilon=1$ and $m_{F_1(\omega,\cdot)}(\cdot,V)$ is defined as in  \eqref{eq: minimointornobordo}.

\begin{lemma}\label{subproc}
    For each $\Sigma\in {\bf M}^{N\times d}$, we have that $\mu_\Sigma$ is a discrete subadditive process.
\end{lemma}
\begin{proof} Let $V\in \mathcal A_0(U)$. Then the  $\mathcal {F}$-measurability of the map $\omega \to \mu_\Sigma(\omega,V)$ follows reasoning as in  \cite[Lemma C.1]{RR}.
By   \normalfont(inc)$_{p^-}$, \normalfont(A0), and \eqref{growth-fomega}, we get that \[\mu_\Sigma(\cdot,V)\le F_1(\cdot, u_{\Sigma}, V) \leq \beta \left( 1+ \beta+\beta|\Sigma|^{p^+} \right)|V|\] so that $\mu_\Sigma(\cdot,V)\in L^1({ \Omega})$.

As $F_1(\omega, \cdot)$ (whenever finite) is an integral functional, for each $\omega$ we have that $\mu_\Sigma(\omega,\cdot)$ is a subadditive  function.
It remains to prove that $\mu_\Sigma$ is stationary.
For each $z\in\Z^d$, $\omega\in\Omega$ and  $V\in \mathcal A_0(U)$, if  $u\in L^1(V; \mathbb{R}^N)$ is  such that $F_1(\tau_z\omega,u, V)<+\infty$ and $u=u_\Sigma$ in a neighborhood of  $\partial V$, then  it holds 
\[
\int_V f(\tau_z\omega,x,Du(x))\,\mathrm{d}x=\int_{V+z} f(\omega,y,D\left(u(y-z)+\Sigma z\right)\,\mathrm{d}y=\int_{V+z} f(\omega,y,Dv(y))\,\mathrm{d}y
\]
where $v(y)=u(y-z)+\Sigma z$  satisfies  $v=u_\Sigma$ in a neighborhood of  $\partial (V+z)$. Therefore
\begin{equation*}
\begin{aligned}
    \mu_\Sigma(\tau_z\omega,V) &=\inf_{u\in L^1(V,\mathbb R^N)} \left\{F_1(\tau_z\omega,u, V)\mid u=u_\Sigma \text{ in a neighborhood of } \partial V\right\}\\ &\geq \inf_{v\in L^1(V+z,\mathbb R^N)} \left\{F_1(\omega,v, V+z)\mid v=u_\Sigma\text{ in a neighborhood of } \partial (V+z)\right\}\\
   &= \mu_\Sigma(\omega,V+z).
\end{aligned}
\end{equation*}
The proof of the converse inequality is analogue.
\end{proof}

The following lemma is substantially \cite[Proposition 1]{dmmodica}. Recall that we say that the subadditive process $\mu_\Sigma$ is ergodic whenever $\tau$ is ergodic.

\begin{lemma}\label{lemma genstochastic}
For each $\Sigma\in {\bf M}^{N\times d}$ there exists a set $\Omega_\Sigma\subseteq \Omega$ of full measure such that
    \begin{equation*}\label{varphi}
    \varphi(\omega,\Sigma)=\lim_{t \to +\infty} \frac{\mu_\Sigma (\omega,tQ)}{t^d|Q|}
\end{equation*}
for every $\omega\in\Omega_\Sigma$ and every cube $Q$ in $\R^d$. Moreover, if $\mu_\Sigma$ is ergodic, then $\varphi$ is independent of $\omega$. 
\end{lemma}
\begin{proof}

    By the Subadditive Ergodic theorem \cite[Theorem 2.7]{AK}, for each $\Sigma\in {\bf M}^{N\times d}$  there exists a set $\Omega_\Sigma\subseteq \Omega$ of full measure such that \begin{equation}\label{limitenaturali}
    \lim_{\substack{t \to +\infty \\ t \in \mathbb{\N}}} \frac{\mu_\Sigma (\omega,tQ)}{t^d |Q|}=\varphi(\omega,\Sigma)
    \end{equation}
    for every $\omega\in\Omega_\Sigma$ and every cube $Q$ with vertices in $\Z^d$.
    It remains to address the problem of extending the limit \eqref{limitenaturali} to real numbers and general cubes.
    Fix $\Sigma\in {\bf M}^{N\times d}$ and $\omega\in\Omega_\Sigma$. Given a cube $Q$, denote by $Q'$ and $Q$ the smallest cube containing $Q$ with vertices in $\Z^d$ and the largest cube contained in $Q$ with vertices in $\Z^d$, respectively. 
For each $t\in\R$, denote by $\lfloor t \rfloor$ and $\lceil t \rceil$ the largest integer smaller than $t$ and the smallest integer larger than $t$, respectively.  By applying the  subadditivity  proved in Lemma \ref{subproc}, we have
    \begin{equation*}
     \begin{aligned}
        \mu_\Sigma(\omega,tQ)
        &\leq \mu_\Sigma(\omega,\lfloor t \rfloor Q)+\mu_\Sigma(\omega,tQ\setminus \lfloor t \rfloor Q)\\
       & \le \mu_\Sigma(\omega,\lfloor t \rfloor Q)+\beta\left(1+\beta+\beta|\Sigma|^{p^+}\right)|tQ\setminus \lfloor t \rfloor Q|,
    \end{aligned}
        \end{equation*}
    and similarly 
    \begin{equation*}
        \mu_\Sigma(\omega,\lceil t \rceil Q)\le \mu_\Sigma(\omega,tQ)+\beta\left(1+ \beta+\beta|\Sigma|^{p^+}\right)|\lceil t \rceil Q\setminus tQ|.
    \end{equation*}
Thus, dividing by $t^d |Q|$ and taking into account that $\lceil t \rceil^{-d} \leq t^{-d} \leq  \lfloor t \rfloor^{-d}$,  we can  apply \eqref{limitenaturali} on $\mu_\Sigma(\omega,\lfloor t \rfloor Q)\lceil t \rceil^{-d} $ and on $\mu_\Sigma(\omega,\lfloor t \rfloor Q)\lfloor t \rfloor^{-d}$, to obtain that,  for every cube $Q$ with vertices in $\Z^d,$ it holds 
    \begin{equation*}
    \lim_{\substack{t \to +\infty \\ t \in \mathbb{\R}}} \frac{\mu_\Sigma (\omega,tQ)}{t^d|Q|}=\varphi(\omega,\Sigma).
\end{equation*}
Since for each cube $Q$ with vertices in $\Q^d$ there exists $s\in\N$ such that $sQ$ has vertices in $\Z^d$, the preceding limit holds also for cubes with vertices in $\Q^d$.
    Given a general cube $Q$ in $\R^d$, for each $\sigma>0$ let $\widehat {Q}_\sigma$ and $Q_\sigma$ be respectively a cube contained in $Q$ with vertices in $\Q^d$ and a cube containing $Q$ with vertices in $\Q^d$ such that $|Q\setminus \widehat {Q}_\sigma|+|Q_\sigma\setminus Q|<\sigma$. Arguing as before, we get that, for every $t>0$, it holds  
\begin{equation*}
        \mu_\Sigma(\omega, t \widehat {Q}_\sigma)-C\sigma t^d\le\mu_\Sigma(\omega, t Q)\le \mu_\Sigma(\omega,t Q_\sigma)+C\sigma t^d
    \end{equation*}
    for some constant $C>0$. Sending first  $t\to +\infty$ and then $\sigma\to 0$, this implies that 
    \begin{equation*}
    \varphi(\omega,\Sigma)=\lim_{t \to +\infty} \frac{\mu_\Sigma (\omega,tQ)}{t^d|Q|}
\end{equation*}
     holds for every cube $Q$ in $\R^d$.  Finally, if  $\tau$ is ergodic, from \cite[Remark, page 59]{AK}, it follows that $\varphi$ is independent of $\omega$.
\end{proof}

\subsection{Stochastic Homogenization in the model case}
In this subsection we will treat the special case when $f(\omega, x, \Sigma)=g(\omega, x,\Sigma)$; hence, throughout this section, for every $V\in \mathcal A(U)$ we consider $\mu_\Sigma(\omega,V)=m_{G_1(\omega,\cdot)}(u_\Sigma,V)$. The main result of this subsection is the following.

\begin{thm}\label{gammamodelomega}  Let  $U\subseteq \mathbb R^d$ be a bounded connected Lipschitz open set and let    $G_\varepsilon$  be given by \eqref{G_varepsilonomega}. Then,  under the hypotheses $(1)$--$(7)$ assumed in this section,  there exists a set $\Omega'\subseteq \Omega$ of full measure and a function $\zeta\colon \Omega\times {\bf M}^{N\times d}\to [0,+\infty)$, measurable  in the product space, such that  for each $\omega\in \Omega'$ 
\begin{enumerate}
\item $\zeta(\omega,\cdot)$ is an $N$-function;
\item $\zeta(\omega,\cdot)$ satisfies the reinforced-$\Delta_2$ property: 
\begin{equation*}\zeta(\omega,A \Sigma)\leq \max\{ \|A\|^{p^+},\|A\|^{p^-}  \} \zeta(\omega,\Sigma) \qquad  \forall A\in {\bf M}^{N\times N},  \forall \Sigma \in {\bf M}^{N\times d}\,;  \end{equation*}
\item the family $G_\varepsilon(\omega,\cdot)$ 
$\Gamma$-converges, with respect to the $L^1(U,\R^N)$ norm, to the functional 
 \begin{align}\label{omegamodel-limit}
G(\omega, u):=
\left\{
\begin{array}{l}
\displaystyle\int_{U} \zeta(\omega, D u)\,dx\quad \mbox{if }u\in W_{\zeta}(U,\R^N), \\
+\infty \quad \mbox{elsewhere in }L^1(U,\R^N)\,.
\end{array}
\right.   
\end{align}
\end{enumerate}
Furthermore, if $u\in W_{\zeta}(U, \R^N)$, then, for every subsequence $\varepsilon_k\to 0$ and every $V\in\A(U)$, we have  
\begin{equation}\label{eq: omegalocal}
G'(\omega, u, V)=G''(\omega, u, V)=\int_V\zeta(\omega, Du)\,dx \qquad \hbox{ for every $\omega\in \Omega'$}.
\end{equation}
Moreover, for every $\omega\in \Omega'$, $\zeta(\omega,\cdot)$ satisfies  the following asymptotic homogenization formula:
\begin{align}\label{zhom}
\zeta(\omega, \Sigma)=\lim_{t\to +\infty}\inf\left\{\frac 1{t^d}\int_{(0,t)^d} g(\omega, x,Du(x))\,\mathrm{d}x\mid u=u_\Sigma \text{ in a neighborhood of } \partial \left((0,t)^d\right)
        \right\}
\end{align}
and, if $\tau$ is ergodic, then $\zeta$ is independent of $\omega$.
\end{thm}

To prove the theorem, we will adapt the steps in Section \ref{convex homog} to the setting at hand. We start collecting some additional properties of the density $\zeta$, which only require minor modifications to \cite{dmmodica}.

\begin{lemma}\label{prop1dalmaso}
    There exists a set $\Omega'\subseteq \Omega$ of full measure such that, for every  $\omega\in\Omega'$ and every $\Sigma\in {\bf M}^{N\times d}$, the limit 
    \begin{equation}\label{eq: zeta}
    \zeta(\omega,\Sigma)=\lim_{t \to +\infty} \frac{\mu_\Sigma (\omega,tQ)}{t^d|Q|}
\end{equation}
is well defined and independent of the cube  $Q$ in $\R^d$.
In particular, $\zeta$ is convex, satisfies \eqref{zhom} and if $\tau$ is ergodic, then $\zeta$ is independent of $\omega$.  Moreover, the function $\zeta$ is measurable on the  product space $\Omega\times {\bf M}^{N\times d}$.  
\end{lemma}
\begin{proof}

For each $\omega\in\Omega$ and  $V\in\mathcal A_0(U)$, we consider the functions $$\Sigma\mapsto  \frac{\mu_\Sigma(\omega,V)}{|V|}.$$
They  are equibounded between 0 and $\beta\left(1+\beta+\beta|\Sigma|^{p^+}\right)$.   We claim that they are also convex. Indeed, fix $\omega\in \Omega$, $V\in\mathcal A_0(U)$,  let  $\lambda\in [0,1]$ and $\Sigma_1,\Sigma_2\in {\bf M}^{N\times d}$. Fixed $\eta>0$, we choose $u_1,u_2$ so that they are admissible in the sense of the definition, respectively, of $\mu_{\Sigma_1}$ and $\mu_{\Sigma_2}$ and satisfy
\begin{equation*}
    G_1(\omega, u_1, V)<\mu_{\Sigma_1}(\omega,V)+\eta, \quad G_1(\omega,u_2, V)<\mu_{\Sigma_2}(\omega,V)+\eta.
\end{equation*} 
By the convexity of $G_1(\omega,\cdot,V)$ we have 
\begin{equation*}
\begin{aligned}
    \mu_{ \lambda\Sigma_1+(1-\lambda)\Sigma_2}(\omega, V)&\le G_1(\omega,\lambda u_1+(1-\lambda) u_2,V)\\ &\le \lambda G_1(\omega,u_1,V)+(1-\lambda)G_1(\omega,u_2,V)\\ &<\lambda\mu_{\Sigma_1}(\omega,V)+(1-\lambda)\mu_{\Sigma_2}(\omega,V)+\eta.
    \end{aligned}
\end{equation*}
Letting $\eta\to 0^+$, the claim is proven.

As a result of convexity and equiboundedness, we have that,   for each fixed $\omega\in\Omega$  and for every fixed cube   $Q\subseteq \R^d,$  the  family $\{\psi_{t,Q}\}_{t>0}$ given by  $$\psi_{t,Q}(\Sigma)= \frac{\mu_\Sigma(\omega,tQ)}{|tQ|} $$ is  locally equicontinuous and equibounded on $ {\bf M}^{N\times d}$ and,     defined   \[Q_t=\{x\in\R^d\mid |x_i|<t, i=1,\dots,d\},\] 
 the function
 \[
    \zeta(\omega,\Sigma)=\limsup_{t\to +\infty} \frac{\mu_\Sigma(\omega,Q_t)}{|Q_t|}\]
 is convex with respect to $\Sigma$. 
Now, for every $\Sigma\in {\bf M_{\Q}}^{d\times N}$ let $\Omega'_\Sigma $ be the set given by  Lemma \ref{lemma genstochastic}. Setting
\begin{equation*}                \Omega'=\bigcap_{\Sigma\in {\bf M_{\Q}}^{d\times N}}\Omega'_\Sigma,
\end{equation*}
we have that  $\mathbb P(\Omega')=1$ and, thanks to Lemma \ref{lemma genstochastic}, it holds
 \begin{equation}\label{suirazio}
  \lim_{t \to +\infty} \frac{\mu_\Sigma (\omega,tQ)}{t^d|Q|}=  \lim_{t \to +\infty} \frac{\mu_\Sigma (\omega,t(-1,1)^d)}{ (2t)^d}= \lim_{t \to +\infty} \frac{\mu_\Sigma(\omega,Q_t)}{|Q_t|}= \zeta(\omega,\Sigma)
\end{equation}
    for each $\omega\in\Omega'$, $\Sigma\in {\bf M_{\Q}}^{N\times d}$, for each cube $Q\subseteq \R^d$.
Now, let   $Q\subseteq \R^d$ be a cube, $\omega\in \Omega'$ and 
$\Sigma\in {\bf M}^{N\times d}$.  Let $t_k\to +\infty$. Up to a subsequence, we have that  $\{\psi_{t_k,Q}\}$   uniformly converges   on the compact subsets of  ${\bf M}^{N\times d}$.
  Let   $\{\Sigma_n\}\subseteq {\bf M_{\Q}}^{N\times d}$  be a sequence  converging to $\Sigma$.
 Then, by exploiting the uniform convergence of $\{\psi_{t_k,Q}\} $, the limit  \eqref{suirazio} and  the continuity of $\zeta(\omega,\cdot)$,  it holds:
 \begin{equation*}
\begin{aligned}
 \lim_{k\to +\infty}\frac{\mu_\Sigma(\omega,t_k Q)}{t_k^d|Q|} &=
\lim_{k\to +\infty} \psi_{t_k,Q} (\Sigma)=\lim_{k\to +\infty} \lim_{n\to +\infty} \psi_{t_k,Q} (\Sigma_n)=\lim_{n\to +\infty} \lim_{k\to +\infty} \psi_{t_k,Q} (\Sigma_n)\\
&= \lim_{n\to +\infty}   \lim_{k\to +\infty}\frac{\mu_{\Sigma_n}(\omega,t_kQ)}{t_k^d|Q|}= \lim_{n \to +\infty} \zeta(\omega,\Sigma_n)= \zeta(\omega,\Sigma). 
\end{aligned}
\end{equation*}
Thanks to the arbitrariness of the sequence $\{t_k\}$, this implies that \eqref{eq: zeta} holds    for each $\omega\in\Omega'$, $\Sigma\in {\bf M}^{N\times d}$ and for every   cube $Q\subseteq \R^d$. In particular,  if $\tau$ is ergodic, from \cite[Remark, page 59]{AK}, it follows that $\zeta$ is independent of $\omega$. Moreover, thanks to \cite[Theorem 2.7]{AK},  the function $\zeta(\cdot, \Sigma)$ is measurable on $\Omega$ for every  $\Sigma\in {\bf M}^{N\times d}$.  Being $\zeta(\omega, \cdot)$  convex and finite (hence continuous), it easily follows that $\zeta$ is measurable on the space product $\Omega\times {\bf M}^{N\times d}$.  Finally, taking $Q=(0,1)^d$ we get \begin{equation*}
\begin{aligned}
\zeta(\omega,\Sigma)
&= \lim_{t\to +\infty}\frac{\mu_\Sigma(\omega,tQ)}{|tQ|} = \lim_{t\to +\infty}\inf\left\{\frac 1{t^d} G_1(\omega,u,tQ)\mid u=u_\Sigma \text{ in a neighborhood of } \partial \left((0,t)^d\right)\right\} \\
&= \lim_{t\to +\infty}\inf \left\{\frac 1{t^d}\int_{(0,t)^d} g(\omega, x,Du(x))\,\mathrm{d}x \mid u=u_\Sigma \text{ in a neighborhood of } \partial \left((0,t)^d\right)\right\}\,,
\end{aligned}
\end{equation*}
that is \eqref{zhom}.
\end{proof}

The $\Gamma$-liminf inequality  is the only point where we need to depart to the proof of the corresponding result in Section \ref{convex homog}, and can be obtained by a typical blow-up argument which we sketch below.
\begin{prop}\label{prop: omegaliminf}
For all $\omega \in \Omega'$, $u \in W^{1,1}(U,\R^N)$ and $V \in \mathcal{A}(U)$,    and for every  sequence  ${\varepsilon_k}\to 0$  it holds
    \begin{equation*}
        \int_V \zeta(\omega, Du)\, dx\le G'(\omega, u, V).
    \end{equation*}
\end{prop}
\begin{proof}
  The proof can essentially be obtained by repeating the computations of \cite[Proposition~4.6]{RS23}, 
which however passes through an approximation of the integrand: the function $g$ is approximated by integrands 
$g_k$ with $L^{q}$-growth for $q<1^*$. We therefore briefly show how the argument can be carried out in our setting, with no approximation, 
by means of a truncation method, indicating only the necessary modifications.  

We fix a sequence $\{u_k\}$ converging to $u$ in $L^{1}(U,\R^N)$  such that
 \[G'(\omega, u, V)=\liminf_{k\to +\infty} G_{\varepsilon_k} (\omega,u_k,V)<+\infty.\]
Let $x_0\in V$ be  a Lebesgue point of $Du$ where  \begin{equation}\label{eq: apdif}
\lim_{r \to 0}\frac{1}{r^d}\int_{Q_r}\left|\frac{u(x)-u(x_0)-D u(x_0)\cdot (x-x_0)}{r}\right|\mathrm{d}x=0\,.
\end{equation}
For fixed $r>0$, we consider $M>\sqrt d|D u(x_0)|+1$ and set $w_{k, r}(x)=(u_k(x)-u(x_0))^{Mr}$ where the truncations  at the level $Mr$  are defined as in   \eqref{troncata} and we omit the dependence on the (fixed) value $M$ for the sake of brevity.
For each positive $0<\eta<1$,  on  the cubes $Q_{\eta r}$ and  $Q_{ r}$  centered in $x_0$, since $g$ is radially increasing and  $|D w_{k, r}(x)|\leq  |Du_{k}(x)|$ for a.e. $x\in U$, we have
\[
G_{\varepsilon_k}(\omega, \eta w_{k, r}, Q_r) \le G_{\varepsilon_k}(\omega, \eta u_k, Q_r) \quad \mbox{ and } \quad G_{\varepsilon_k}(\omega, \eta  w_{k, r} , Q_r\setminus Q_{\eta r}) \le G_{\varepsilon_k}(\omega, \eta u_k, Q_r \setminus Q_{\eta r})\,.
\]
With this, we can perform the same construction as in \cite{RS23}, using $w_{k, r}$ in place of $u_k$, but when using the growth condition from above we need to show that
\begin{equation}\label{eq: daprovare1}
\lim_{r \to 0}\lim_{k\to +\infty}\frac{1}{r^d}\int_{Q_r}\left|\frac{w_{k, r}(x)-D u(x_0)\cdot (x-x_0)}{r}\right|^{p^+}\mathrm{d}x=0.
\end{equation}

As $w_{k, r} \to u_r:=(u(x)-u(x_0))^{Mr}$ in every $L^p$, to get \eqref{eq: daprovare1} it suffices   to show
\begin{equation}\label{eq: daprovare2} 
\lim_{r \to 0}\frac{1}{r^d}\int_{Q_r}\left|\frac{u_{r}(x)-D u(x_0)\cdot (x-x_0)}{r}\right|^{p^+}\mathrm{d}x=0\,.
\end{equation}
We note that, for a.e. $x\in Q_r$, it holds
\begin{equation}\label{stima1}
 \left|\frac{u_{r}(x)-D u(x_0)\cdot (x-x_0)}{r}\right| \leq 2 M+|D u(x_0)|\frac{|x-x_0|}{r} \leq 2M+\sqrt{d}|D u(x_0)|\leq 3M. 
\end{equation}
Moreover, we set $K_r:=\{ x \in Q_r: \, |u(x)-u(x_0)|\ge Mr\}$ and observe that, thanks to the choice of $M$,  by construction,   we have
\[
\left|\frac{u(x)-u(x_0)-D u(x_0)\cdot (x-x_0)}{r}\right|\ge  M- |D u(x_0)|\sqrt{d}> 1\, \quad \hbox{ for a.e. }x \in K_r.
\]
Hence, using \eqref{eq: apdif}, we get
\[
\begin{split}
\lim_{r \to 0}\frac1{r^d}\int_{K_r}\left|\frac{u_{r}(x)-D  u(x_0)\cdot (x-x_0)}{r}\right|^{p^+}\mathrm{d}x\le (3M)^{p^+}\lim_{r \to 0}\frac{|K_r|}{r^d}\\
\le (3M)^{p^+}\lim_{r \to 0}\frac{1}{r^d}\int_{Q_r}\left|\frac{u(x)-u(x_0)-D u(x_0)\cdot (x-x_0)}{r}\right|\mathrm{d}x=0\,.
\end{split}
\]
Instead,  since $u_r(x)= u(x)-u(x_0)$ for a.e.  $x\in Q_r \setminus K_r$,   by taking into account also \eqref{stima1},   on  $Q_r \setminus K_r$   it holds
\[
\begin{split}
\left|\frac{u_{r}(x)-D u(x_0)\cdot (x-x_0)}{r}\right|^{p^+}\le (3M)^{p^+-1}\left|\frac{u_{r}(x)-D u(x_0)\cdot (x-x_0)}{r}\right|\\
=(3M)^{p^+-1}\left|\frac{u(x)-u(x_0)-D u(x_0)\cdot (x-x_0)}{r}\right|\,.
\end{split} 
\]
Hence
\[
\begin{split}
&\lim_{r \to 0}\frac{1}{r^d}\int_{Q_r\setminus K_r}\left|\frac{u_{r}(x)-D u(x_0)\cdot (x-x_0)}{r}\right|^{p^+}\mathrm{d}x\\
&\le (3M)^{p^+-1}\lim_{r \to 0}\frac{1}{r^d}\int_{Q_r}\left|\frac{u(x)-D u(x_0)\cdot (x-x_0)}{r}\right|\mathrm{d}x=0\,,
\end{split}
\]
so that we get \eqref{eq: daprovare2} and we can now conclude repeating the argument  in~\cite[Proposition~4.6]{RS23}.
\end{proof}

As for the $\Gamma$-limsup inequality, one has to be a little bit cautious. We need indeed to apply Lemma \ref{limsup} and the two sided estimate \eqref{eq:convminima}, which hold in principle along a subsequence satisfying \eqref{epsk}, that
possibly depends on $\omega$. We start by stating a result up to taking a subsequence.

\begin{prop}\label{gamma=varphi}
Fix $\omega\in\Omega'$. Then the  function $\zeta(\omega,\cdot)$, given by \eqref{eq: zeta}, is an $N$-function and satisfies the reinforced-$\Delta_2$ property. Furthermore, there exists a subsequence $\varepsilon_{k_j}$ such that 
\begin{equation}
    G''(\omega,u,V)=\Gamma(L^1)\text{-} \limsup_{j\to \infty}G_{\varepsilon_{k_j}}(\omega,u,V)\le \int_V \zeta(\omega,Du)\, \mathrm{d}x
\end{equation} 
for each $u\in W_{\zeta}(U,\R^d)$ and each $V\in\A(U)$.  
\end{prop}
\begin{proof}
    Fix a cube $Q$ in $\R^d$, $\Sigma\in {\bf M}^{N\times d}$, $t\in\R$ and $\omega\in\Omega'$. By a scaling argument we have
    \begin{equation}\label{perpriscalaminimi}
        \mu_\Sigma(\omega,tQ)=t^d m_{G_{1/t}(\omega,\cdot)}(u_\Sigma,Q).
    \end{equation}
By Remark \ref{no-periodicity}, we can apply all the results going from Lemma \ref{boundsp+} to Proposition \ref{G''min} to the family $G_\varepsilon(\omega, \cdot, \cdot)$. We fix a subsequence $\varepsilon_{k_j}$ such that 
\begin{equation}\label{epskgen}
     G'(\omega,\cdot,D)=G''(\omega, \cdot, D) \quad \hbox{ for every } D \in \mathcal D
 \end{equation}
and we do not further relabel the sequence $\varepsilon_k$ for brevity. Following the steps of Section \ref{convex homog},  for every  $V\in\A(U)$ we represent $G''(\omega,\cdot,V)$, that is the $\Gamma\text{-}\limsup$ of the family $G_{\varepsilon_k}(\omega,\cdot, V)$, as \begin{equation*}\label{reprstocG}
    G''(\omega, u, V)=\int_V\gamma(\omega,x,Du)\,\mathrm{d}x
\end{equation*}
for every $u\in W^{1,p^+}(U,\R^N)$; here $\gamma(\omega, \cdot, \cdot)$ is given by Lemma \ref{limsup}. We want to show that $\gamma(\omega,\cdot,\cdot)$ agrees with $\zeta(\omega,\cdot)$, which entails independence of $x$ (and of $\omega$, if $\tau$ is ergodic).
For $\omega\in \Omega'$,  we use Lemma \ref{G''min} on the function $\gamma(\omega,\cdot,\cdot)$ as follows: the lower bound in \eqref{eq:convminima},   combined with \eqref{perpriscalaminimi} and Lemma \ref{prop1dalmaso}, implies that, for every fixed  $\Sigma\in {\bf M}^{d\times N}$,   there exists a negligible set $N_{\omega, \Sigma}$ such that for every  $x\in U\setminus N_{\omega, \Sigma} $ it holds 
    \begin{equation*}
    \begin{aligned}
        \gamma(\omega,x,\Sigma)&=\lim_{r\to 0^+}\frac{m_{G''(\omega,\cdot)}(u_\Sigma, Q_r(x))}{|Q_r(x)|}\\ &\ge \limsup_{r\to 0^+}
\limsup_{k\to +\infty}
\frac{m_{G_{\varepsilon_k}(\omega,\cdot)}(u_\Sigma, Q_r(x))}{|Q_r(x)|}\\ &=\limsup_{r\to 0^+}
\limsup_{k\to +\infty}  \frac{\mu_\Sigma\left(\omega,\varepsilon_k^{-1} Q_r(x)\right)}{\varepsilon_k^{-d} |Q_r(x)|}\\ &=\zeta(\omega,\Sigma).
\end{aligned}
    \end{equation*}
    Similarly, exploiting the upper bound in   \eqref{eq:convminima}, we also obtain 
    \begin{equation*}
        \gamma(\omega,x,\Sigma)\le\zeta(\omega,\Sigma)
    \end{equation*}
for every $x\in U\setminus N_{\omega, \Sigma}$. An equicontinuity argument shows that the set of full measure where equality holds can be taken  independent of $\Sigma$.  Thus, $\gamma(\omega,x,\cdot)$ and $\zeta(\omega,\cdot)$ are equal a.e. $x\in U$.    Since  $\gamma(\omega, \cdot,\cdot)$ does not depend on $x$ on a set of full measure,   the same arguments in Proposition \ref{gammaN} and Corollary \ref{necessario} entail the remaining properties.
\end{proof}

We are now in a position to prove Theorem \ref{gammamodelomega}.

\begin{proof}[Proof of Theorem~\ref{gammamodelomega}]
By  \eqref{bassop-}, one clearly has $G'(\omega, u, U)=+\infty$ whenever $u \notin W^{1, p^-}(U,\R^d)$.  The properties of $\zeta$  and the representation \eqref{omegamodel-limit} follow from Lemma \ref{prop1dalmaso}, Propositions \ref{prop: omegaliminf} and \ref{gamma=varphi} with $V=U$, upon using the Urysohn property of $\Gamma$-convergence to get rid of subsequences depending on $\omega$. This is possible since $\zeta$ does not depend on the chosen subsequence. Furthermore, by repeating this argument on each Lipschitz subset of $U$, we deduce that, in particular, \eqref{epskgen} is satisfied by {\it any } subsequence $\varepsilon_k$ without further extraction. Hence, all the results in  Proposition \ref{G''min} hold with no need of taking a subsequence, so does also Proposition \ref{gamma=varphi}. Combining the latter with Proposition \ref{prop: omegaliminf}, we get \eqref{eq: omegalocal} and conclude the proof.
\end{proof}

\subsection{Stochastic Nonconvex Homogenization}
We eventually turn to the general case. Throughout this section, we consider the the sets $\Omega_\Sigma$  and the function $\varphi(\omega,\Sigma)$ defined by Lemma \ref{lemma genstochastic}
in $\Omega'\times{\bf M}_{\Q}^{N\times d}$, where $\Omega'=\bigcap_{\Sigma \in {\bf M}_{\Q}^{N\times d}} \Omega_\Sigma$. Clearly $\Omega'$ has full measure and by possibly taking the intersection with another set of full measure we can assume that Theorem \ref{gammamodelomega} holds for $\omega \in \Omega'$. In our main result, one of the points to be checked is that actually the function $\varphi(\omega,\Sigma)$ can be uniquely extended to a Carath\'eodory function on  $\Omega'\times{\bf M}^{N\times d} $.
We state and prove our main result.
\begin{thm}\label{main}
Let  $U\subseteq \mathbb R^d$ be a bounded connected Lipschitz open set and let $F_\varepsilon$  be given by \eqref{casostoc}. Let $\zeta$ be the function given by \eqref{zhom}.
Then, there exists a set $\Omega'\subseteq \Omega$ of full measure and  a Carath\'eodory function  $\varphi\colon \Omega\times {\bf M}^{N\times d}\to [0,+\infty)$   such that,  for each $\omega \in \Omega'$, the family $F_\varepsilon(\omega,\cdot)$ 
$\Gamma$-converges, with respect to the $L^1(U,\R^N)$ norm, to the functional 
 \begin{align*}
F(\omega, u):=
\left\{
\begin{array}{l}
\displaystyle\int_{U} \varphi(\omega, D u)\,\mathrm{d}x\quad \mbox{if }u\in W_{\zeta}(U,\R^N), \\
+\infty \quad \mbox{elsewhere in }L^1(U,\R^N)\,.
\end{array}
\right.   
\end{align*}
Moreover, $\varphi$ satisfies  the following asymptotic homogenization formula:
\begin{align}\label{eq: cellvarphi}
\varphi(\omega, \Sigma)=\lim_{t\to +\infty}\inf\left\{\frac 1{t^d}\int_{(0,t)^d} f(\omega, x,Du(x))\,\mathrm{d}x\mid u=u_\Sigma \text{ in a neighborhood of } \partial \left((0,t)^d\right)
        \right\}
\end{align}
for all $\Sigma \in {\bf M}_{\Q}^{N\times d}$ and, if $\tau$ is ergodic, then $\varphi$ is independent of $\omega$.
\end{thm}
\begin{proof}
We first observe that \eqref{eq: cellvarphi} immediately follows from Lemma \ref{lemma genstochastic} (taking $Q=(0,1)^d$), as well as the fact that  $\varphi$ is independent of $\omega$ whenever $\tau$ is ergodic.
Taking also into account Remark \ref{no-periodicity2}, we initially fix a subsequence $\varepsilon_{k_j}$ (possibly depending on $\omega$) so that Theorem \ref{mainthmnonconvex} holds. For this subsequence, there exists a function $\eta(\omega,x,\Sigma)$ which is rank $1$-convex with respect to $\Sigma$ and has $p^+$-growth from above such that 
$F_{\varepsilon_{k_j}}(\omega,\cdot)$ 
$\Gamma$-converges, with respect to the $L^1(U,\R^N)$ norm, to the functional 
 \begin{align*}
F(\omega, u):=
\left\{
\begin{array}{l}
\displaystyle\int_{U} \eta(\omega, x, D u)\,\mathrm{d}x\quad \mbox{if }u\in W_{\zeta}(U,\R^N), \\
+\infty \quad \mbox{elsewhere in }L^1(U,\R^N)\,.
\end{array}
\right.   
\end{align*}
Observe that in particular, by rank $1$-convexity and local boundedness, $\eta(\omega, x, \cdot)$ is  equicontinuous locally on ${\bf M}^{N\times d} $. Now, applying \eqref{eq:convminimaF} and arguing as in Proposition \ref{gamma=varphi} we get that
\[
\varphi(\omega,\Sigma)=\eta(\omega,x,\Sigma)
\]
for a.e. $x\in U$ and for every $(\omega,\Sigma)\in \Omega'\times{\bf M}_{\Q}^{N\times d}$. The equicontinuity of $\eta$ implies now that $\varphi(\omega,\Sigma)$ can be uniquely extended to a Carath\'eodory function on  $\Omega'\times{\bf M}^{N\times d} $, and that the above equality holds for all $\Sigma$. With this, and the Urysohn property of $\Gamma$-convergence we conclude that the desired result holds for the whole family $F_\varepsilon$.
\end{proof}

\noindent\textbf{Funding}
The work of  F.S.~is part of the GNAMPA-INdAM, Project 2026:  “Sistemi multi-agente e replicatore: derivazione particellare e ottimizzazione” ({\tt CUP E53C25002010001}). 

\medskip

\noindent\textbf{Competing interests.} The authors declare no competing interests.

\medskip 

\noindent\textbf{Data availability.} There are no data sets associated with this paper.

\medskip

\noindent\textbf{Author contibutions.} The authors contributed equally to the manuscript.

\end{document}